\documentclass[english,reqno,11pt]{amsart}
\usepackage{amsmath, amsthm, amsfonts}
\usepackage{hyperref} 
\usepackage{hyperref}
\hypersetup{
	colorlinks=true,
		citecolor=blue!60!black,
		linkcolor=red!60!black,
		urlcolor=green!40!black,
		filecolor=yellow!50!black,
	breaklinks=true,
	pdfpagemode=UseNone,
	bookmarksopen=false,
}
\usepackage{amsmath,amsthm,amssymb}
\usepackage{amscd,indentfirst,epsfig}
\usepackage{latexsym}
\usepackage{times}
\usepackage{enumerate}
\usepackage{mathrsfs}
\usepackage{stmaryrd}
\usepackage{amsopn}
\usepackage{amsmath}
\usepackage{amssymb,dsfont,mathtools}
\usepackage{amsfonts,bm}
\usepackage{amsbsy,amsmath}
\usepackage{amscd}
\usepackage{xcolor}

\usepackage[mono=false]{libertine}
\usepackage[T1]{fontenc}
\usepackage{amsthm}
\usepackage[normalem]{ulem} 
\usepackage[cal=euler, scr=boondoxo]{}
\usepackage{microtype}

\usepackage{numprint}

\begin{document}

\makeatletter
\def \leq {\leqslant}
\def \le {\leq}
\def \geq {\geqslant}
\def\e{\varepsilon}
\def \ge {\geq}
\def\tend{\rightarrow}
\def\R{\mathbb R}
\def\S{\mathbb S}
\def\Z{\mathbb Z}
\def\N{\mathbb N}
\def\C{\mathscr{C}}
\def\D{\mathcal D}
\def\T{\mathcal T}
\def\E{\mathcal E}
\def\s{\sigma}
\def\k{\kappa}
\def \a{\beta}
 \def\be{\beta}
\def\t{\theta}
\def\l{\ell}
\def \M {\mathcal{M}}
\def\L{\Lambda}
\def\O{\Omega}
\def\g{\gamma}
\def \ds {\displaystyle}
\def\G{\Gamma}
\def \mM {\mathfrak{m}}
\def\f{\varphi}
\def\o{\omega} 
\def \d {\mathrm{d}}
\def \lm {\bm{m}}
\def \lM {\mathds{M}}
\def \lD {\mathds{D}}
\def\U{\Upsilon}
\def\z{\zeta}
\def \Q {\mathcal{Q}}
\def\over{\bm}
\def\b{\backslash}
\def\fet{f_{\ast}}
\def \get{g_{\ast}}
\def\fprim{f^{\prime}}
\def\fprimet{f^{\prime}_\ast }
\def\vet{v_{\ast}}
\def  \pa {\partial}
\def \var {\dd}
\def\vprim{v^{\prime}}
\def\vprimet{v^{\prime}_{\ast}}
\def\grad{\nabla}

\def\Log{\textrm{Log }}
\newtheorem{theo}{Theorem}[section]
\newtheorem{prop}[theo]{Proposition}
\newtheorem{cor}[theo]{Corollary}
\newtheorem{lem}[theo]{Lemma}
\newtheorem{hyp}[theo]{Assumptions}
\newtheorem{defi}[theo]{Definition}
\newtheorem{rmq}[theo]{Remark}
\def \vr {\vartheta}
\def \up {\Upsilon}
\DeclarePairedDelimiter\lnorm{\llbracket}{\rrbracket}		

\def \dd {\bm{\varepsilon}}

\def \ind {\mathbf{1}}
\numberwithin{equation}{section}
 
\title[The Landau equation with moderate soft potentials]{The Landau equation with moderate soft potentials: An approach using $\varepsilon$-Poincar\'e inequality and Lorentz spaces}

\author{R. Alonso}
\address{Texas A\&M University at Qatar, Division of Arts and Sciences, Education City--Doha, Qatar.}
\email{ricardo.alonso@qatar.tamu.edu}

\author{V. Bagland}
\address{Universit\'{e} Clermont Auvergne, LMBP, UMR 6620 - CNRS,  Campus des C\'ezeaux, 3, place Vasarely, TSA 60026, CS 60026, F-63178 Aubi\`ere Cedex,
France.}\email{Veronique.Bagland@uca.fr}

\author{B.  Lods}
\address{Universit\`{a} degli
Studi di Torino \& Collegio Carlo Alberto, Department of Economics, Social Sciences, Applied Mathematics and Statistics ``ESOMAS'', Corso Unione Sovietica, 218/bis, 10134 Torino, Italy.}\email{bertrand.lods@unito.it}

\maketitle
 
\begin{abstract} This document presents an elementary approach using $\varepsilon$-Poincar\'e inequality to prove generation of $L^{p}$-bounds, $p\in(1,\infty)$, for the homogeneous Landau equation with moderate soft potentials $\gamma\in[-2,0)$.  The critical case $\gamma=-2$ uses an interpolation approach in the realm of Lorentz spaces and entropy.  Alternatively, a direct approach using the Hardy-Littlewood-Sobolev (HLS) inequality and entropy is also presented.   On this basis, the generation of pointwise bounds $p=\infty$ is deduced from a De Giorgi argument.\\

\end{abstract}

%\tableofcontents

\section{Introduction}
The spatially homogeneous Landau equation describes the evolution of a density distribution $f:=f(t,v)\geq0$ of a plasma of particles having velocity $v\in \R^d$ at a time $t>0$.  It is given by
\begin{equation}\label{LFD1}
\partial_{t} f(t,v) =  \Q(f,f)\,,
%\\= {\grad}_v \cdot \int_{\R^{d}} |v-\vet|^{\g+2} \, \Pi(v-\vet) 
%\Big\{f(t,\vet) \grad f(t,v) - f(t,v) {\grad f}(t,\vet) \Big\}
%\, \d\vet\,,\\
\qquad t\geq0, \quad v \in \R^{d}\,,
\end{equation}
where the collision operator is given by
\begin{equation}\label{eq:Landau}
\Q(f,f)(v):= {\grad}_v \cdot \int_{\R^{d}} |v-\vet|^{\g+2} \, \Pi(v-\vet) 
\Big\{ \fet \grad f - f {\grad f}_\ast \Big\}
\, \d\vet\, ,
\end{equation}
with the usual shorthand $f:= f(v)$, $f_* := f(v_*)$ and $\Pi(z)=\mathrm{Id}-\frac{z \otimes z}{|z|^{2}}$.  The Landau equation \eqref{LFD1} is supplemented with initial condition
\begin{equation}\label{eq:Init}
f(t=0,v)=f_{\mathrm{in}}(v), \qquad v \in \R^{d}.
\end{equation}
Results regarding \textit{a priori} energy estimates involving $L^{p}$-norms, mainly in the moderately soft potential case $- 2 \leq \gamma < 0$, are revisited.

\smallskip
\noindent
The spirit of this paper combines techniques of \cite{Wu} and \cite{GG} bringing a more direct proof and improved estimates of those of \cite{Wu} which include generation of norms with specific rates and better long time behaviour.  It also includes the pointwise estimation, the case $p=\infty$, of solutions for both norm creation and propagation using a De Giorgi's type of approach, see \cite{degiorgi} for an original implementation in the classical parabolic context.  In the kinetic context of the spatially homogeneous Boltzmann equation refer to \cite{ricardo} and for the homogeneous Landau-Fermi-Dirac to \cite{ABDL1}.  The main tools used are the $\varepsilon$-Poincar\'e inequality and an interesting interpolation in Lorentz spaces that deal with the critical case of $\gamma=-2$.  Lorentz interpolation is used to circumvent the problematic application of the HLS inequality in $L^1-L^p$ case.  As noticed in \cite{Wu}, entropy propagation can be added to the argument to properly control the size of the most singular term.  In addition, we present an alternative proof to the aforementioned approach using Hardy-Littlewood-Sobolev inequality in a way that avoids the $L^1-L^{p}$ case.  This later approach involves fewer interpolation steps giving explicit constants in the control of the most singular term.  Entropy is also needed to control the size of such term. 

\smallskip
\noindent
We mention that an alternative approach to obtain pointwise bounds is developed in \cite{silvestre} using techniques for parabolic equations in non-divergence form.  The technique is used in the context of classical solutions of the equation using a quantitative maximum principle approach in the spirit of Aleksandrov-Bakelman-Pucci.  The estimates obtained using this method are quite sharp, yet solutions are assumed in the classical sense.  This is an important difference with the spirit of the techniques brought here that work in the realm of weak solutions which presumably simplify the rigorous implementation of an approximation scheme for the equation.

\subsection{Notations} Define the Lebesgue space $L^{p}_{s}(\R^{d})$, with $s \in \R $ and $ p\in[1,\infty)$, through the norm
$$
\displaystyle \|f\|_{L^p_{s}} := \left(\int_{\R^{d}} \big|f(v)\big|^p \, 
\langle v\rangle^s \, \d v\right)^{\frac{1}{p}}, \qquad L^{p}_{s}(\R^{d}) :=\Big\{f\::\R^{d} \to \R\;;\,\|f\|_{L^{p}_{s}} < \infty\Big\}\,,
$$
where $\langle v\rangle :=\sqrt{1+|v|^{2}}$. 
%More generally, for any weight function $\varpi\::\:\R^{d} \to \R^{+}$ define
%$$
%L^{p}(\varpi) :=\Big\{f\,:\,\R^{d} \to \R\;\Big|\;\|f\|_{L^{p}(\varpi)}^{p}:=\int_{\R^{d}}\big|f\big|^{p}\,\varpi\,\d v  < \infty\Big\}\,,\qquad p \in [1,\infty)\,.
%$$
%With this notation one can write, for example, $L^{p}_{s}(\R^{d})=L^{p}\big(\langle \cdot \rangle^{s}\big)$ with $p \in [1,\infty)$ and $s \geq 0$.

% \smallskip
% \noindent
% Define the weighted Sobolev spaces by 
% $$
% W^{k,p}_{s}(\R^{d}) :=\Big\{f \in L^{p}_{s}(\R^{d})\;\Big|\;\partial_{v}^{\beta}f \in L^{p}_{s}(\R^{d}) \:\;\forall\, |\beta| \leq k\Big\}\quad\text{with}\quad k \in \N\,,
% $$
% with the standard norm
% $$
% \|f\|_{W^{k,p}_{s}} := \bigg( \sum_{0\leq |\beta| \leq k} 
% \int_{\R^{d}} \big| \partial^{\beta}_v f(v)\big|^p\, 
% \langle v\rangle^s \, \d v\bigg)^{\frac{1}{p}},
% $$
% with multi-index  $\a=(i_1,\cdots,i_d) \in \N^d$, $|\a|=\sum^{d}_{k=1} i_k$ and 
% $\partial^{\beta}_v f =\partial_1^{i_1}\cdots\partial_d^{i_d} f$.  In the Hilbert case $p=2$ we simply write $H^{k}_{s}(\R^{d})=W^{k,2}_{s}(\R^{d})$ with $k \in \N$ and $s \geq 0$. 

%
\begin{defi} \label{defi:fin}
We say that $f \in \mathcal{Y}_{0}(f_{\mathrm{in}})$ if $f\in L^{1}_{2}(\R^{d})$ and 
\begin{equation}\label{eq0}
\int_{\R^{d}}f(v)\left(\begin{array}{c}1 \\v \\ |v|^{2}\end{array}\right)\d v=\int_{\R^{d}}f_{\mathrm{in}}(v)\left(\begin{array}{c}1 \\v \\ |v|^{2}\end{array}\right)\d v=:\left(\begin{array}{c}\varrho  \\ \varrho u \\ d\varrho \theta +\varrho|u|^{2}\end{array}\right),
\end{equation}
and $H(f) \leq H(f_{\mathrm{in}})$
where
$$
H(f)=\int_{\R^{d}}f(v)\log f(v)\d v\,.
$$
\end{defi}
\noindent
By a simple scaling argument, there is no loss in generality to assume that 
\begin{equation}\label{eq:Mass}
\varrho =\theta =1, \qquad u=0\,.
\end{equation}
We introduce for $i,j\in\{1,\cdots,d\}$
\begin{equation*}
\left\{
\begin{array}{rcl}
a(z) & = & \left(a_{i,j}(z)\right)_{i,j} \quad \mbox{ with }
\quad a_{i,j}(z) 
= |z|^{\g+2} \,\left( \delta_{i,j} -\frac{z_i  z_j}{|z|^2} \right),\medskip\\
 b_i(z) & = & \sum_k \partial_k a_{i,k}(z) = -(d-1) \,z_i \, |z|^\g,  \medskip\\
 c(z) & = & \sum_{k,l} \partial^2_{kl} a_{k,l}(z) = -(d-1) \,(\g+d) \, |z|^\g. \\
\end{array}\right.
\end{equation*}
For any $f \in L^{1}_{2+\g}(\R^{d})$, we define the matrix-valued mapping  $\mathcal{A}[f]$ by
$$
\mathcal{A}[f]=\big(\mathcal{A}_{ij}[f]\big)_{ij}:=\big(a_{ij}\ast f\big)_{ij}\,.
$$
In the same way, define the vector-valued mapping $\bm{b}[f]$ and the scalar mapping $\bm{c}_{\g}[f]$ given by 
$$
\bm{b}[f] = \big(\bm{b}_{i}[f]\big)_i = \big(b_{i} \ast f\big)_i\,, \quad\text{and}\quad \bm{c}_{\g}[f](v)=\big(c \ast f\big)(v)\,.
$$
The dependency with respect to the parameter $\g$ in $\bm{c}_{\g}[f]$ is stressed since, in several places, we apply the same definition with $\g+1$ replacing $\g$.

\smallskip
\noindent
With these notations, the Landau equation can then be written alternatively under the form
\begin{equation}\label{LFD}
\left\{
\begin{array}{ccl}
\;\partial_{t} f &= &\grad \cdot \big(\,\mathcal{A}[f]\, \grad f
- \bm{b}[f]\, f\big)=\mathrm{Tr}\left(\mathcal{A}[f] D^{2}f\right) - \bm{c}_{\g}[f]f , \medskip\\
\;f(t=0)&=&f_{\mathrm{in}}\,.
\end{array}\right.
\end{equation}
where $D^{2}f$ is the Hessian matrix of $f$ (for two matrices $A=(A_{i,j})$, $B=(B_{i,j})$, $\mathrm{Tr}(AB)=\sum_{i,j}A_{i,j}B_{j,i}$).\\  

Let us make precise the notion of weak solution we consider in this paper.
\begin{defi}
 Consider $-2\le \g<0$. Let $f_{\mathrm{in}}\in L^1_2(\R^d)\cap L\log L(\R^d)$. We say that $f$ is a weak solution to the Cauchy problem \eqref{LFD} if the following conditions are fullfilled:
  \begin{enumerate}
  \item $f\ge 0$ with $f\in\C([0,\infty);\D'(\R^d))\cap L^\infty ([0,\infty); L^1_2\cap L\log L(\R^d))$.
  \item For any $t\ge 0$ $$\int_{\R^d} f(t,v)\varphi(v)\,\d v=\int_{\R^d} f_{\mathrm{in}}(v)\varphi(v)\,\d v \qquad \mbox{ for }\quad \varphi(v)=\big\{ 1,v,|v|^2 \big\} $$
    and $$\int_{\R^d} f(t,v)\log f(t,v) \,\d v\le \int_{\R^d} f_{\mathrm{in}}(v)\log f_{\mathrm{in}}(v)\, \d v\,.$$
  \item For any $\varphi\in \C^1( [0,\infty);\C^\infty_{c}(\R^d))$ and $t\ge 0$,
    \begin{multline*}
      \int_{\R^d} f(t,v)\varphi(t,v) \,\d v= \int_{\R^d} f_{\mathrm{in}}(v) \varphi(0,v) \,\d v +\int_0^t \int_{\R^d} f(\tau,v)\partial_t\varphi(\tau,v) \, dv \,\d\tau \\
      + \frac12 \sum_{i,j=1}^d \int_0^t \d\tau\int_{\R^d\times \R^d} f(\tau,v) f(\tau,\vet)\, a_{i,j}(v-\vet)\, \big[\partial^2_{i,j} \varphi(\tau,v) +\partial^2_{i,j} \varphi(\tau,\vet) \big] \,\d v \,\d \vet \\
      +\sum_{i=1}^d \int_0^t \d\tau\int_{\R^d\times \R^d} f(\tau,v) f(\tau,\vet) \, b_{i}(v-\vet)\,\big[\partial_{i} \varphi(\tau,v) -\partial_{i} \varphi(\tau,\vet) \big] \,\d v \,\d \vet\,.
      \end{multline*}
  \end{enumerate} 
\end{defi}
In the case of moderately soft potentials $\g\in[-2,0)$, C. Villani proved in \cite{villH} the existence of global weak solutions to \eqref{LFD}. Some useful properties of these solutions are gathered in Appendix~\ref{sec:known}.

\subsection{The role of the $\e$-Poincar\'e inequality}

In the work \cite{GG}, the role played in showing $L^{p}$ \textit{a priori} bounds by a variant of the so-called Poincar\'e inequality has been emphasised. It is common, to derive energy estimates for solution to \eqref{LFD}, to investigate for smooth functions $f=f(v)$ the properties of the inner product 
$$
-\langle f,\Q(f)\rangle= - \int_{\R^{d}}f\Q(f)\d v=\int_{\R^{d}}\mathcal{A}[f]\,\nabla f\,\cdot \nabla f \d v+  {\frac12}\int_{\R^{d}}\bm{c}_{\g}[f]\,f^{2}\d v\,.
$$
The function $-\bm{c}_{\g}[f] \geq0$ acts as a potential interacting with the diffusive part involving $\mathcal{A}[f]$. In particular, the coercivity of $-\Q(f)$ 
amounts to 
$$
\int_{\R^{d}} \mathcal{A}[f]\,\nabla f\cdot\nabla f\d v \geq - {\frac12} \int_{\R^{d}}\bm{c}_{\g}[f]\,f^{2}\d v.
$$
A weaker inequality would be enough to deduce a control of the growth of solutions to \eqref{LFD}, say
$$
\int_{\R^{d}}\mathcal{A}[f]\,\nabla f\cdot\nabla f\d v + C\,\int_{\R^{d}}f^{2}\d v \geq - {\frac12} \int_{\R^{d}}\bm{c}_{\g}[f]\,f^{2}\d v.
$$
A refined version of such inequality has been obtained for $-2 < \g <0$ in \cite{GG} showing that, for smooth $\phi$ and $f \in \mathcal{Y}_{0}(f_{\rm in})$, for any $\e \in (0,1)$ there exist $\Lambda(\e) >0$ such that
\begin{equation}\label{eq:epoincare}
-\int_{\R^{d}}\bm{c}_{\g}[f]\,\phi^{2}\d v \leq \e \int_{\R^{d}}\mathcal{A}[f]\nabla \phi\cdot\nabla \phi \d v + \Lambda(\e)\int_{\R^{d}}\phi^{2}\d v.
\end{equation}
The proof of \eqref{eq:epoincare} in \cite{GG} uses tools from harmonic analysis, namely Muckenhoupt $A_{p}$-weights.  
%In the study of a variant of \eqref{LFD1} in a quantum setting, we revisited such an inequality in our recent contribution \cite{ABDL1} providing an elementary proof of it for $-2 < \gamma < 0.$  
We present here a path to show inequality \eqref{eq:epoincare} using classical estimates of the Riesz potential in Lorentz spaces deriving a version that holds up to the critical case $\g=-2$, see Proposition \ref{propo:GG2}.  In fact, we actually provide an alternative proof of the inequality using Hardy-Littlewood-Sobolev inequality.  In both procedures the key is to find a way around the problematic $L^{1}-L^{p}$ case to the HLS inequality and use entropy to control the size of the most singular term.

\smallskip
\noindent %and recently extended to the inhomogeneous framework in \cite{AMSY} for hard potentials.  We stress that 
%Since the work \cite{ABDL1} was dealing with the Landau-Fermi-Dirac equation, we resume here the arguments given there to the classical Landau case but, more importantly, we extend the result to the critical case $$\gamma=-2.$$

\subsection{Main results}
The following version of the $\e$-Poincar\'e inequality holds for weak solutions to \eqref{LFD}.
\begin{theo}\label{cor:MAIN} Let $-2\le\g<0$ and $T>0$ and a nonnegative initial datum $f_{\rm in} \in L^{1}_{2} (\R^{d})\cap L\log L(\R^{d})$ be given satisfying \eqref{eq:Mass} and consider any  {global weak solution} $f(t,v)$ to \eqref{LFD} with initial datum $f_{\rm in}$. %\textcolor{purple}{If $\g=-2$, we assume also that $f_{\rm in} \in L^{1}_{\kappa}(\R^{d})$ for some $\kappa >\kappa^{\dagger}$}.
  Then, for any $\e>0$,  there exists $C >0$ such that for any smooth function $\phi$,
\begin{equation}\label{eq:poin0}
-\int_{\R^{d}}\phi^{2}\,\bm{c}_{\g}\big[ f(t) \big]\,\d v \leq \e\,\int_{\R^{d}}\left|\nabla \left(\langle v\rangle^{\frac{\g}{2}}\phi(v)\right)\right|^{2}\d v \\
+C\int_{\R^{d}}\phi^{2}\langle v\rangle^{\g}\d v, \quad \forall t\in[0,T]\,,
\end{equation}
where $C$ depends on $\e$, $\g$, $f_{\rm in}$ and $T$. 
Moreover, if $\g\in(-2,0)$, one has $C=C_0\left(1+\e^{\frac{\g}{2+\g}}\right)$ where $C_0$ only depends on $\|f_{\rm in}\|_{L^{1}_{2}}$.
\end{theo}

 Let us note that, in the above theorem, the initial datum is only assumed to satisfy natural (physical) bounds. {In the case $\gamma=-2$, the de la Vall\'ee Poussin Theorem \cite{Lau} is used in order to avoid extra assumptions on moments of the initial datum. The $\varepsilon$-Poincar\'e inequality result is stated on a finite time interval. For $-2<\g<0$, it actually holds on $[0,\infty)$. For $\g=-2$, it can be extended to $[0,\infty)$ by assuming some additional finite moments for the initial condition. Indeed, combining the linear growth of moments and the large time behavior enables then to control additional moments uniformly in time. This is what we need in order to extend this result to  $[0,\infty)$.}
 \medskip

 {We subsequently use the above $\e$-Poincar\'e inequality in combination with a De Giorgi argument to show the derivation of pointwise bounds for solutions to \eqref{LFD}.  Such bound was the scope of the work \cite{GG} for the case $-2< \g <0$ providing an estimate of the form
$$
f(v,t) \leq C_{0}\left(1+\frac{1}{t}\right)^{\frac{d}{2}}\,\left(1+|v|\right)^{\frac{d}{2}|\g|}\,.
$$
The estimates presented here includes $\gamma=-2$ and will not deteriorate as the velocity increases. This allows us to recover results obtained in \cite{silvestre} which, by methods borrowed from the nonlinear parabolic equations theory, L. Silvestre derived for $-2 \leq \gamma < 0$ the appearance of pointwise bounds for strong solutions to \eqref{LFD}. We provide here a different approach which seems more direct and does not require any knowledge about parabolic equations.. We mention that for the homogeneous Landau-Fermi-Dirac equation, a quantum correction of the Landau equation, pointwise estimates have been obtained in \cite{ABDL1} using similar techniques (for $\gamma>-2$). Typically our main pointwise estimates can be stated as follows;
\begin{theo}\label{theo:Pointwise} Let $-2\leq \g<0$ and $T>0$. Let a nonnegative initial datum $f_{\mathrm{in}} \in L^1_2(\R^d)\cap L\log L(\R^d)$ satisfying \eqref{eq:Mass} be given. Let  $f(t,\cdot)$ be a weak-solution to \eqref{LFD}. Let us assume that $f_{\rm in}\in L^1_{s}(\R^d)$ for some $s >\frac{d}{2}|\g|$. Then, for any $T > t_{*} >0$, there exists $C_{T} >0$ depending on $T >0$ and $f_{\rm in}$ such that
\begin{equation}\label{eq:Linf}
\sup_{t \in [t_{*},T)}\left\|f(t)\right\|_{L^{\infty}} \leq C_{T}\max\left(1,t_{*}^{-\beta}\right),
\end{equation}
for some explicit $\beta>0$ depending on $s,d,\g$.\end{theo}
A precise statement is given in Theorem \ref{Linfinito*} for $-2 < \g <0$ and Theorem \ref{Linfinito*-2} for $\g=-2$ where the precise expression of the constant $C_{0}$ and the parameter $\beta >0$ can be found.}  {As said, the method of proof is based upon the implementation of a De Giorgi type argument \cite{degiorgi}. The method relies into two main steps:
\begin{itemize}
\item First, we use the $\e$-Poincar\'e inequality to show the appearance of suitable $L^{p}$-bound, $p < \infty$ (see Theorem \ref{theo:boundedL2} for a precise statement). 
\item Second, a modification of the original De Giorgi argument proposed in \cite{ricardo} for the study of the Boltzmann equation without cut-off and implemented in \cite{ABDL1} for the Landau-Fermi-Dirac equation, allows to show that the $L^{p}$-bounds derived in the previous steps can be upgraded into $L^{\infty}$-bounds. Such a typical $L^{p}-L^{\infty}$ argument in the spirit of \cite{degiorgi}  has already proved to be useful, even in the context of spatially inhomogeneous kinetic equations \cite{guo,AMSY}. 
\end{itemize} }
\subsection{Related literature}  The study of the spatially homogeneous Landau equation has a long history and the case of hard potentials, corresponding to $\g \geq 0$ has been thoroughly  investigated in \cite{DeVi1} showing existence, uniqueness, appearance of smoothness and of moments. The convergence towards equilibrium is also known to hold at some exponential rate and similar results hold true for the case of Maxwell molecules corresponding to $\g=0.$ 

Regarding the case of moderately soft potentials $\g \in [-2,0)$, the situation is not too different as far as the Cauchy theory is concerned and the main differences are related to:
\begin{enumerate}[(i)]
\item moments which are no longer created but are only propagated with a growth which is at most linear in time.
\item Convergence to equilibrium which is not exponential anymore but only algebraic. Notice that  "stretched exponential" convergence rate can be proven under some restrictive assumptions on the initial datum \cite{kleber,ABDL1}.
\end{enumerate} 
We refer to \cite{Wu, DesvJFA, DeVi2} for a precise description of the existing results as well as the work \cite{ALL} for a complete review of the Landau equation with moderately soft potentials which are the main object of the present paper. 

The situation changes drastically in the case of very soft potentials for which $-d \leq \g <-2$ (including the Coulomb case $\g=-d$) for which the Cauchy theory is much more involved.  Indeed, only weak solutions (including H-solutions) are known to exist whereas uniqueness is an open problem \cite{DesvJFA}. Notice that stretched exponential convergence to equilibrium still holds \cite{CDH}. As far as regularity of solutions is concerned, it is possible to get a bound on the Hausdorff dimension of the times in which the solution might be singular \cite{GGIV,GIV}. Various perturbative results are also available in the literature and in particular local-in-time solutions for large data and global-in-time solutions for data close to equilibrium have been recently obtained in \cite{DesvHe}.  In the case $-d \leq \g < -2$, the global well-posedness of \eqref{LFD} for general initial data  is still an open question and, in the most recent years, several important \emph{conditional results} have been derived. We refer the reader to \cite{ABDL2} for a more complete description of the literature in this case as well as with new conditional results under some Prodi-Serrin like criterion.

\subsection{Organization of the paper} The proof of Theorem  \ref{cor:MAIN}  is given in Section \ref{sec:Poin} where the two different cases $-2 < \g <0$ and $\g=-2$ are treated separately. The consequences of the above $\e$-Poincar\'e inequality are then given in Section \ref{sec:DeG}. The appearance of $L^{p}$-norm for $1 < p < \infty$ are given in Section \ref{sec:Lp} resulting in Theorem \ref{theo:boundedL2}. The $L^{p}-L^{\infty}$ argument of De Giorgi is then given in Section \ref{sec:level} with a full proof of Theorem \ref{theo:Pointwise} where again we distinguish the two cases $-2 < \g <0$ and $\g=-2$. In Appendix \ref{sec:known} we recall some known results about weak solutions to the Landau equation whereas Appendix \ref{sec:Lorentz} recalls some results about Lorentz spaces used in the core of the text.

\subsection*{Acknowledgments} B.L. gratefully acknowledges the financial support from the
Italian Ministry of Education, University and Research (MIUR), Dipartimenti di Eccellenza
grant 2022-2027 as well as the support from the de Castro Statistics Initiative, Collegio Carlo Alberto (Torino). The authors are grateful to L. Desvillettes for  insightful discussions about the Landau equation.

\section{The $\e$-Poincar\'e inequality extended}\label{sec:Poin}
We recall here inequality \eqref{eq:epoincare} in the known case $-2 < \g <0$ and provide the extended proof valid for the critical case $\g=-2$. The proof uses elementary knowledge of Lorentz spaces whose main properties are given in Appendix \ref{sec:Lorentz}.
\subsection{Case $-2 < \g <0$} We first show the $\e$-Poincar\'e inequality in the case $-2 < \g <0.$
\begin{prop}\label{prop:GG} Assume that $-2 < \g < 0$ and let a nonnegative mapping $f_{\mathrm{in}} \in L^1_2(\R^d)\cap L\log L(\R^d)$ satisfying \eqref{eq:Mass} be given. {There exists $C_{0} >0$ depending only on $\|f_{\rm in}\|_{L^{1}_{2}}$ such that  for any $\e>0$, any $f \in \mathcal{Y}_{0}(f_{\mathrm{in}})$ and any smooth function $\phi$, }
\begin{equation}\label{eq:estimat}
-\int_{\R^{d}}\phi^{2}\bm{c}_{\g}[f]\d v \leq \e\,\int_{\R^{d}}\left|\nabla \left(\langle v\rangle^{\frac{\g}{2}}\phi(v)\right)\right|^{2}\d v 
+C_{0}(1+\e^{\frac{\g}{2+\g}})\int_{\R^{d}}\phi^{2}\langle v\rangle^{\g}\d v.
\end{equation}
\end{prop}
\begin{proof} Let $f \in \mathcal{Y}_{0}(f_{\mathrm{in}})$ be fixed. For a given nonnegative $\phi$ set
$$
I[\phi] :=-\int_{\R^{d}}\phi^{2}\bm{c}_{\g}[f]\d v=(d-1)\,(\g+d)\int_{\R^{d}\times \R^{d}}|v-\vet|^{\g}\phi^{2}(v)f(\vet)\d v\d\vet\,.
$$
For any $v,\vet \in \R^{d}$, if $|v-\vet| < \frac{1}{2}\langle v\rangle$, then $\langle v\rangle\leq 2\langle \vet\rangle$, and
we deduce from this, see \cite[Eq. (2.5)]{amuxy}, that
$$|v-\vet|^{\g}\leq 2^{-\g}\langle v\rangle^{\g}\left(\ind_{\left\{|v-\vet| \geq\frac{\langle v\rangle}{2}\right\}}+
\langle \vet\rangle^{-\g}|v-\vet|^{\g}\ind_{\left\{|v-\vet|< \frac{\langle v\rangle}{2}\right\}}\right).$$
Therefore,
\begin{equation}\label{eq:ineq}
I[\phi] \leq  2^{-\g}(d-1)\,(\g+d)\left(I_{1}+I_{2}\right)
\end{equation}
with
\begin{multline*}
I_{1}:=\int_{\R^{d}}\langle v\rangle^{\g}\phi^{2}(v)\d v \int_{|v-\vet| \geq \frac{\langle v\rangle}{2}} f(\vet)\d\vet\,,\\
\text{and}\qquad I_{2}:=\int_{\R^{d}}\langle \vet\rangle^{-\g}f(\vet)\d \vet\int_{|v-\vet| < \frac{1}{2}\langle v\rangle }|v-\vet|^{\g}\langle v\rangle^{\g}\phi^{2}(v)\d v.\end{multline*}
Setting $\psi(v)=\langle v\rangle^{\frac{\g}{2}}\phi(v)$ and $F(v)=\langle v\rangle^{-\g}f(v)$ one has that
\begin{equation}\label{eq:ineqI1}
I_{1} \leq \|f_{\rm in}\|_{L^{1}}\|\psi\|_{L^{2}}^{2}\,,
\end{equation}
and
$$
I_{2} \leq \int_{\R^{d}\times \R^{d}}|v-\vet|^{\g}F(\vet)\psi^{2}(v)\d v\d\vet=:J\,.
$$ 
We observe that
$$
J=\int_{\R^{d}}\mathcal{I}_{d+\gamma}[F]\,\psi^{2}(v)\d v,
$$
where, for any $\alpha \in (0,d)$, $\mathcal{I}_{\alpha}[g]$ denotes the Riesz operator
$$
\mathcal{I}_{\alpha}[g](v)=\int_{\R^{d}}\frac{g(\vet)}{|v-\vet|^{d-\alpha}}\d \vet, \qquad \alpha \in (0,d)\,.
$$
Then, using the H\"older inequality in Lorentz spaces it holds that
$$
J  \leq \left\|\mathcal{I}_{d+\g}[F]\right\|_{q,\infty}\,\left\|\psi^{2}\right\|_{p,1}, \qquad \frac{1}{q}+\frac{1}{p}=1\,.
$$
Recall that $\mathcal{I}_{\alpha}$ is a bounded operator from $L^{1}$ to $L^{q,\infty}$ if $1=\frac{\alpha}{d}+\frac{1}{q}$ which means that, choosing $q=\frac{d}{|\g|}$, $p=\frac{d}{d+\g},$
there exists $C_1=C_1(d,\g)$ such that
$$
\left\|\mathcal{I}_{d+\g}[F]\right\|_{ \frac{d}{|\g|},\infty} \leq C_1(d,\g)\left\|F\right\|_{L^{1}}
$$
and
$$
J \leq C_1(d,\g)\,\|F\|_{L^{1}}\,\|\psi^{2}\|_{{p,1}}=C_1(d,\g)\,\|F\|_{L^{1}}\,\|\psi\|_{{2p,2}}^{2}, \qquad p=\frac{d}{d+\g}\,.
$$
Since $-2< \g <0$, note that $2 < 2p < \frac{2d}{d-2}$.  Then, using Young's inequality together with \eqref{eq:interpopq} shows that for any $\delta >0$, 
$$
{C_1(d,\g)}\|\psi\|_{2p,2}^{2} \leq  {C_2({d,\g})}\delta^{-\frac{1-\theta}{\theta}}\|\psi\|_{{2,2}}^{2} + \delta\,\|\psi\|_{{\frac{2d}{{d-2}},2}}^{2}
$$
where $\frac{1}{p}=\theta+(1-\theta)\frac{ {d-2}}{d}$, that is $\theta=\frac{2+\g}{2}$. Since $L^{2,2}=L^{2}$, this means that
$$
J \leq C_2(d,\g)\delta^{\frac{\g}{2+\g}}\|F\|_{L^{1}}\|\psi\|_{L^{2}}^{2}\,+\delta\,\|F\|_{L^{1}}\,\|\psi\|_{{\frac{2d}{d-2},2}}^{2}\,.
$$
Furthermore, Sobolev inequality in Lorentz spaces, see Theorem \ref{theo:Sob}, gives that
\begin{equation}\label{eq:precised}
\|\psi\|_{{\frac{2d}{d-2},2}} \leq C_{d}\,\|\nabla \psi\|_{L^{2}}\,.
\end{equation}
Consequently, it follows that 
\begin{equation}\label{eq:JJ}
J  \leq C_2(d,\g)\delta^{\frac{\g}{2+\g}}\|F\|_{L^{1}}\|\psi\|_{L^{2}}^{2}\,+\delta\,C_{d}\,\|F\|_{L^{1}}\,\|\nabla \psi\|^{2}_{L^{2}}, \qquad \forall \delta >0\,.
\end{equation}
Notice that $\|f_{\rm in}\|_{L^{1}} \leq \|F\|_{L^{1}} \leq \|f_{\rm in}\|_{L^{1}_{2}}$ since $-\g < 2$.  Putting together with \eqref{eq:ineq} and \eqref{eq:ineqI1}, after choosing $\e=2^{-\g}(d-1)\,(\g+d)C_{d}\|F\|_{L^{1}}\delta$, there exists $C=C(d,\g, {\|f_{\rm in}\|_{L^{1}_{2}}}) >0$ such that
$$
I[\phi] \leq C\,\e^{\frac{\g}{\g+2}}\,\|\psi\|^{2}_{L^{2}}+\e\,\|\nabla \psi\|_{L^{2}}^{2} +   {2^{-\g}(d-1)\,(\g+d)}\|f_{\rm in}\|_{L^{1}_{2}}\|\psi\|^{2}_{L^{2}}, \qquad \forall \e >0\,,
$$
which is the desired result recalling that $\psi=\langle \cdot\rangle^{\frac{\g}{2}}\phi$.\end{proof}
\subsection{Case $\g=-2$} This is a critical case for which the aforementioned Sobolev inequality does not leave a necessary room to introduce a small parameter $\e$. To overcome this issue, we resort to Lorentz space and fully exploits the fact that entropy is bounded.  %

\begin{prop}\label{propo:GG2} Let $\g=-2$ and let a nonnegative mapping $f_{\mathrm{in}} \in L^1_2(\R^d)\cap L\log L(\R^d)$ satisfying \eqref{eq:Mass} be given.  We assume that there exists $\bm{\Psi}\in\C([0,\infty))$ such that
    \begin{equation}\label{eq:mpsi}
      m_{\bm{\Psi}}(f):=\int_{\R^d}f(v)\,\bm{\Psi}\!\left(| v|^2\right)\d v<+\infty
      \end{equation}
    and
    \begin{equation}\label{eq:limpsi}
      \lim_{r\to\infty}\frac{\bm{\Psi}(r)}{r}=+\infty.
      \end{equation}
  Then, for any $\e>0$, there exists $C >0$ depending only on $\e$, $\|f\|_{L^{1}_{2}}$, an upper bound of $m_{\bm{\Psi}}(f)$ and an upper bound of $H(f)$ such that
for any  smooth function $\phi$,
  \begin{equation}\label{eq:estimatc}
-\int_{\R^{d}}\phi^{2}\bm{c}_{\g}[f]\d v \leq \e\int_{\R^{d}}\left|\nabla \left(\langle v\rangle^{\frac{\g}{2}}\phi(v)\right)\right|^{2}\d v 
+C \int_{\R^{d}}\phi^{2}\langle v\rangle^{\g}\d v\,.
\end{equation}
%where, for any $s \geq0$, the statistical moment is defined as
%$$
%\lm_{s}(f):=\int_{\R^{d}}\langle v\rangle^{s}f(v)\d v\,.
%$$
\end{prop}

\begin{proof}%[Proof using Lorentz spaces] 
The proof is quite similar to the aforementioned for $\g \in (-2,0)$. In particular, estimates \eqref{eq:ineq} and \eqref{eq:ineqI1} hold.  Moreover, still with the notations
$$
F(v)=\langle v\rangle^{-\g}f(v)=\langle v\rangle^{2}f(v), \qquad \psi(v)=\langle v\rangle^{\frac{\g}{2}}\phi(v)=\langle v\rangle^{-1}\phi(v)\,,
$$
and
$$
J = \int_{\R^{d}\times \R^{d}}|v-\vet|^{-2}\psi(v)^{2}F(\vet)\d \vet\d v\,,
$$
the computations previously performed show that
$$
J \leq C_{d}\|F\|_{L^{1}}\,\|\psi\|_{L^{2p,2}}^{2}, \qquad p=\tfrac{d}{d+\g}=\tfrac{d}{d-2}\,.
$$
Now we are in the situation in which the exponent $2p=\frac{2d}{d-2}$ is precisely the Sobolev exponent that, using again \eqref{eq:precised}
$$
J \leq C_{d}\|F\|_{L^{1}}\,\|\nabla \psi\|_{L^{2}}^{2}\,,
$$
which gives
$$
I_{-2}[\phi]=-\int_{\R^{d}}\phi^{2}\bm{c}_{-2}[f]\d v \leq C_{d}\left(\|F\|_{L^{1}}\,\|\nabla \psi\|_{L^{2}}^{2}+\|f\|_{L^{1}}\|\psi\|^{2}_{L^{2}}\right)\,.
$$
To create room here and make the coefficient in front of $\|\nabla \psi\|_{L^{2}}^{2}$ small we use the entropy.  Indeed, split the integral defining $J$ according to $|v-\vet| >1$ and $|v-\vet| \leq 1$ to get
$J=J_{1}+J_{2}$ with
$$J_{1}:=\int_{|v-\vet|>1}|v-\vet|^{-2}F(\vet)\psi^{2}(v)\d v\d\vet \leq \|F\|_{L^{1}}\|\psi^{2}\|_{L^{1}} \leq \|f\|_{L^{1}_{2}}\|\psi\|_{L^{2}}^{2}\,,
$$
and
$$J_{2}:=\int_{|v-\vet|\leq 1}|v-\vet|^{-2}F(\vet)\psi^{2}(v)\d v\d\vet\,.$$
For any $R_1 >0$, one further splits
$$F=F^{+}_{R_1}+F^{-}_{R_1}, \qquad F^{+}_{R_1}=F\ind_{F >R_1}, \qquad F^{-}_{R_1}=F\ind_{F \leq R_1}\,,$$
therefore,
$$
J_{2} \leq J^{+}_{2,R_1}+J^{-}_{2,R_1}\,,
$$
with 
$$
J_{2,R_1}^{\pm}:=\int_{|v-\vet|\leq1}F^{\pm}_{R_1}(\vet)|v-\vet|^{-2}\psi^{2}(v)\d v\d\vet\,.
$$
Observe that
$$
J_{2,R_1}^{-} \leq R_1\int_{|v-\vet|\leq 1}\psi^{2}(v)|v-\vet|^{-2}\d v\d\vet\leq R_1\|\psi\|^{2}_{L^{2}}\sup_{v}\int_{|v-\vet|\leq 1}|v-\vet|^{-2}\d \vet\,,
$$
where this last integral is independent of $v$ and given by
$$
\int_{|v-\vet|\leq 1}|v-\vet|^{-2}\d \vet=|\S^{d-1}|\int_{0}^{1}r^{d-3}\d r=\frac{|\S^{d-1}|}{d-2}\,.$$
Thus,
$$
J_{2,R_1}^{-}\leq C_{d}\,R_1\,\|\psi\|^{2}_{L^{2}}\,.
$$
We estimate now the most singular term $J_{2,R_1}^{+}$. To do so, we use the estimate in Lorentz spaces as 
$$
J_{2,R_1}^{+} \leq \int_{\R^{d}}\mathcal{I}_{d-2}[F^{+}_{R_1}]\,\psi^{2}\d v \leq C_{d}\|\mathcal{R}_{d-2}[F^{+}_{R_1}]\|_{L^{\frac{d}{2},\infty}}\,\|\psi^{2}\|_{L^{\frac{d}{d-2},1}}\,,
$$
so that,
$$J_{2,R_1}^{+}\leq C_{d}\|F^{+}_{R_1}\|_{L^{1}}\,\|\psi\|_{L^{\frac{2d}{d-2},2}}\leq \tilde{C}_{d}\|F^{+}_{R_1}\|_{L^{1}}\,\|\nabla \psi\|_{L^{2}}^{2}.$$
Combining all this estimates, we end up now with
\begin{equation}\label{eq:I-2}
I_{-2}[\phi] \leq C_d \|f\|_{L^{1}_{2}}\|\psi\|_{L^{2}}^{2}+ C_{d}R_1\|\psi\|_{L^{2}}^{2}+ \tilde{C}_{d}\|F^{+}_{R_1}\|_{L^{1}}\,\|\nabla \psi\|_{L^{2}}^{2}, \qquad \forall R_1 >0\,.
\end{equation}
{Using $m_{\bm{\Psi}}(f)$ and the entropy estimate the term $\tilde{C}_{d}\|F_{R_1}^{+}\|_{L^{1}}$ is small as $R_1$ is sufficiently large.  Indeed, for any $R_2>0$, we have
\begin{equation*}\begin{split}
  \|F^{+}_{R_1}\|_{L^{1}}& = \int_{\{F>R_1, f\ge\sqrt{R_1}\}}f(v)\langle v\rangle^2\d v+\int_{\{F>R_1, f<\sqrt{R_1}\}} f(v)\langle v \rangle^2 \d v\\
  & \le  (1+R_2)\int_{\{F>R_1, f\ge\sqrt{R_1}, | v|^2\le R_2\}}f(v)\d v+\int_{\{|v|^2> R_2\}}f(v)(1+|v|^2)\d v \\
  & \hspace{3cm}+\int_{\{\langle v\rangle^2 > \sqrt{R_1}\}}f(v)\langle v\rangle^2 \d v\,.
\end{split}\end{equation*}
Since $f$ satisfies \eqref{eq:mpsi} with a bounded energy and a bounded entropy, we obtain  
\begin{multline}\label{eq:F+R}
  \|F^{+}_{R_1}\|_{L^{1}} \le 2\;\frac{1+R_2}{\log(R_1)}\,H(f)+ \frac1{R_2} \int_{\R^d} f(v)|v|^2\d v \\
  + m_{\bm{\Psi}}(f)\;\sup_{r\ge R_2} \frac{r}{{\bm{\Psi}}(r)} +m_{\bm{\Psi}}(f)\; \sup_{r\ge \sqrt{R_1}-1} \frac{1+r}{{\bm{\Psi}}(r)} \,.
\end{multline}
We then deduce from \eqref{eq:limpsi} the existence of $R_2>0$ such that
$$ \tilde{C}_{d}\left(\frac1{R_2} \int_{\R^d} f(v)|v|^2\d v 
+ m_{\bm{\Psi}}(f)\;\sup_{r\ge R_2} \frac{r}{{\bm{\Psi}}(r)} \right)\le \frac{\e}{2},$$
and then the existence of $R_1>0$ such that
$$ \tilde{C}_{d}\left( 2\;\frac{1+R_2}{\log(R_1)}\,H(f) + m_{\bm{\Psi}}(f)\; \sup_{r\ge \sqrt{R_1}-1} \frac{1+r}{{\bm{\Psi}}(r)} \right)\le \frac{\e}{2}.$$
It thus follows from \eqref{eq:I-2} and \eqref{eq:F+R} that \eqref{eq:estimatc} holds with $C=C_d\|f\|_{L^1_2}+C_dR_1$. }
\end{proof}
\begin{rmq}
  If $f\in L^1_s(\R^d)$ with $s>2$, on can of course choose ${\bm{\Psi}}(r)=r^{\frac{s}{2}}$ in Proposition \ref{propo:GG2}. In such a case, one checks that
  $$\sup_{r\ge R_2}\frac{r}{{\bm{\Psi}}(r)}=R_2^{\frac{2-s}2} \: \text{ and } \;
%  and for $R_1\ge4$,
 \sup_{r\ge \sqrt{R_1}-1} \frac{1+r}{{\bm{\Psi}}(r)}\le \sup_{r\ge \sqrt{R_1}-1} \frac{2r}{{\bm{\Psi}}(r)}= 2(\sqrt{R_1}-1)^{\frac{2-s}2}$$
for any $R_{1} \ge4.$  Therefore, in the above proof, one can take 
  $$R_2= \max\left(\frac{4 \tilde{C}_{d}}{\e} \int_{\R^d}f(v)|v|^2\d v, \left(\frac{4 \tilde{C}_{d}}{\e}\; m_{\bm{\Psi}}(f)\right)^{\frac{2}{s-2}}\right),$$
  and
  $$ R_1=\max\left(4,\exp\left(\frac{8 \tilde{C}_{d}(1+R_2)}{\e}\,H(f)\right), \left(\left(\frac{8 \tilde{C}_{d}}{\e} \;m_{\bm{\Psi}}(f)\right)^{\frac{2}{s-2}}+1\right)^2\right).$$
  \end{rmq}
One deduce then the following version of the $\e$-Poincar\'e inequality   valid for solutions to the Landau equation \eqref{LFD}.
\begin{cor}\label{cor:epsilon_poincare}
 Let $\g=-2$, $T>0$ and let a nonnegative initial datum $f_{\mathrm{in}} \in L^1_2(\R^d)\cap L\log L(\R^d)$ satisfying \eqref{eq:Mass} be given.  Let $f(t,\cdot)$ be a weak solution to \eqref{LFD}. 
  Then, for any $\e>0$, there exists $C >0$ depending only on $\e$, $f_{\mathrm{in}}$ and $T$ such that, for  any  smooth function $\phi$ and any $t\in [0,T]$ 
    \begin{equation}\label{eq:estimatc_sol}
-\int_{\R^{d}}\phi^{2}(v)\bm{c}_{\g}[f](t,v)\d v \leq \e\int_{\R^{d}}\left|\nabla \left(\langle v\rangle^{\frac{\g}{2}}\phi(v)\right)\right|^{2}\d v 
+C \int_{\R^{d}}\phi^{2}(v)\langle v\rangle^{\g}\d v\,.
\end{equation}
\end{cor}
\begin{proof}
Since $| \cdot |^2 \in L^1(\R^d,f_{\rm in}(v)\d v)$, we deduce from the de la Vall\'ee Poussin Theorem (see \cite[Theorem 8]{Lau}) that there exists a  convex function ${\bm{\Psi}}\in\C^\infty([0,+\infty))$ such that ${\bm{\Psi}}(0)={\bm{\Psi}}'(0)=0$, ${\bm{\Psi}}'$ is a concave function, ${\bm{\Psi}}'(r)>0$ for $r> 0$,
  $$\lim_{r\to+\infty}\frac{{\bm{\Psi}}(r)}{r}=\lim_{r\to+\infty} {\bm{\Psi}}(r)=+\infty, $$
  and
  \begin{equation}\label{eq:Poussin}
    \int_{\R^d}f_{\mathrm{in}}(v)\,{\bm{\Psi}}(| v|^2)\,\d v<\infty.
  \end{equation}
  Let us show that the bound \eqref{eq:Poussin} propagates over time for the associated solution $f(t,\cdot)$ to \eqref{LFD}. Mutliplying \eqref{LFD} with $ {\bm \Psi}(| v|^2)$ and integrating over $\R^d$, we get, after some integrations by parts, 
  \begin{equation}\label{eq:mom_evol}
    \dfrac{\d}{\d t}\int_{\R^{d}}f(t,v)\, {\bm \Psi}(|v|^{2})\d v=4\int_{\R^{d}}\int_{\R^{d}}f\,f_{\ast}|v-\vet |^{-2}\,\Lambda_{{\bm \Psi}}(v,\vet )\,\d \vet \d v ,
    \end{equation}
where 
\begin{equation*}
\Lambda_{ {\bm \Psi}}(v,\vet )=\frac{d-1}{2}\left[|\vet |^{2}-|v|^2\right] {\bm \Psi}'(|v|^{2})+\left[|v|^{2}|\vet |^{2}-(v \cdot \vet )^{2}\right]\, {\bm \Psi}''(|v|^{2}).\end{equation*}
First, a symmetry argument combined with the convexity of $ {\bm \Psi}$ enables to obtain that
\begin{equation*}\begin{split}
  &   \int_{\R^{d}}\int_{\R^{d}}f\,f_{\ast}|v-\vet |^{-2}\, \left[|\vet|^2-|v|^2\right] {\bm \Psi}'(|v|^{2}) \,\d \vet \d v\\
  &  \hspace{2cm} = \int_{\R^{d}}\int_{\R^{d}}f\,f_{\ast}|v-\vet |^{-2}\, \left[|v|^2-|\vet|^2\right]   {\bm \Psi}'(|\vet|^{2})\, \d v \d \vet\\
  &  \hspace{2cm} = - \frac12 \int_{\R^{d}}\int_{\R^{d}}f\,f_{\ast}|v-\vet |^{-2}\, \left(|v|^2-|\vet|^2\right)  \left[ {\bm \Psi}'(|v|^{2})- {\bm \Psi}'(|\vet |^{2})\right] \, \d v \d \vet \le 0.
%  &  \hspace{2cm} \le 0
  \end{split}\end{equation*}
Thus, \eqref{eq:mom_evol} becomes
$$ \dfrac{\d}{\d t}\int_{\R^{d}}f(t,v)\, {\bm \Psi}(|v|^{2})\d v \le 4\int_{\R^{d}}\int_{\R^{d}}f\,f_{\ast}|v-\vet |^{-2}\,\left[|v|^{2}|\vet |^{2}-(v \cdot \vet )^{2}\right]\,  {\bm \Psi}''(|v|^{2})\,\d \vet \d v. $$
Now, since $ {\bm \Psi}'$ is concave, $ {\bm \Psi}''$ is nonincreasing. This implies, with the convexity of $ {\bm \Psi}$, that
$$0\le  {\bm \Psi}''(r)\le  {\bm \Psi}''(0)\qquad \mbox{ for }\quad r\ge0.$$
The above inequality with the estimate
$0\le |v|^2|\vet|^2-(v\cdot\vet)^2\le |v|^2 \,|v-\vet|^2$
leads to 
\begin{equation}
  \dfrac{\d}{\d t}\int_{\R^{d}}f(t,v)\, {\bm \Psi}(|v|^{2})\,\d v \le  4  {\bm \Psi}''(0) \int_{\R^{d}}\int_{\R^{d}}f\,f_{\ast}\,|v|^2\,\d \vet\,\d v 
 \le   4  {\bm \Psi}''(0) \|f_{\mathrm{in}}\|_{L^1}  \|f_{\mathrm{in}}\|_{L^1_2}.  
\end{equation}
Consequently, for any $t\in[0,T]$,
$$\int_{\R^{d}}f(t,v)\, {\bm \Psi}(|v|^{2})\d v \le \int_{\R^{d}}f_{\mathrm{in}} (v)\, {\bm \Psi}(|v|^{2})\d v+4\, T\, {\bm \Psi}''(0) \|f_{\mathrm{in}}\|_{L^1}  \|f_{\mathrm{in}}\|_{L^1_2} .$$
Applying Proposition \ref{propo:GG2} to $f(t,\cdot)$ with $t\in[0,T]$ then completes the proof.  
\end{proof}
{\begin{rmq}
    If we assume that $f_{\mathrm{in}}\in L^1_s(\R^d)\cap L\log L(\R^d)$ for some $s>2$, then, denoting by $f$ a global weak solution to \eqref{LFD} associated to $f_{\mathrm{in}}$, we deduce from Theorem \ref{theo:prop_mom} that
    $$ \sup_{t\in [0,T]} \int_{\R^d} f(t,v) |v|^s \d v<\infty.$$
    In this case,  Proposition \ref{propo:GG2} immediately implies that \eqref{eq:estimatc_sol} holds. 
\end{rmq}}

The above results allow to prove the main result in the Introduction:

\begin{proof}[Proof of Theorem \ref{cor:MAIN}] The proof of Theorem \ref{cor:MAIN} is a simple consequence of the $\e$-Poincar\'e in the various cases. The case $\g=-2$ corresponds to Corollary \ref{cor:epsilon_poincare}. If $-2 < \g<0$, the result just comes from Proposition \ref{prop:GG} and the fact that solution $f(t) \in \mathcal{Y}_{0}(f_{\rm in})$ for any $t \in[0,T].$ \end{proof}

{We end this Section dedicated to the $\e$-Poincar\'e inequality by a simpler proof of Proposition \ref{propo:GG2} in the case $d=3$. We state the result in the simplified case in which $f \in L^{1}_{s}(\R^{d})$ with $s >2$. The alternative proof is based upon the Hardy-Littlewood-Sobolev inequality in dimension $d=3$.}
\begin{prop}
    Assume that $\g=-2$, $d=3$  and let a nonnegative mapping $f_{\mathrm{in}} \in L^1_2(\R^d)\cap L\log L(\R^d)$ satisfying \eqref{eq:Mass} be given. There exist $C_{0} >0$ depending only on $\|f_{\rm in}\|_{L^{1}}$ and $C_{1}>0$ depending only on $\|f_{\rm in}\|_{L^{1}_{2}}$ and $H(f_{\rm in})$ such that, for any $s >2$, any $\e >0$ and any $f \in \mathcal{Y}_{0}(f_{\mathrm{in}})$,
    \begin{multline*}
-\int_{\R^{3}}f^{2}\bm{c}_{\g}[f]\d v \leq \e\int_{\R^{3}}\left|\nabla \left(\langle v\rangle^{\frac{\g}{2}}f(v)\right)\right|^{2}\d v 
\\
+C_{0}\left(1+\exp\left(C_{1} \,\e^{-\frac{2}{s-2}}\,m_{s}(f)^{\frac{2}{s-2}}\right)\right)\int_{\R^{3}}f^{2}\langle v\rangle^{\g}\d v\,,
\end{multline*}
    where, for any $s \geq0$, the statistical moment $m_{s}(f)$ is defined as
$$
m_{s}(f):=\int_{\R^{3}}\langle v\rangle^{s}f(v)\d v\,.
$$
\end{prop}
\begin{proof} Being the proof quite similar to that of Prop. \ref{propo:GG2}, let us describe only how HLS inequality provides an alternative way to estimate the most singular term.  Define, for any function $\phi$, 
\begin{equation*}
J_{R} := \int_{|v-\vet|\leq1}f\ind_{f >R}(\vet)|v-\vet|^{-2}\phi^{2}(v)\d v\d\vet,\qquad R\geq1\,.
\end{equation*}
Applying then  Hardy-Littlewood-Sobolev \cite[Theorem 4.3]{LL} with $1 < p, r < \infty$, it holds
$$J_{R} \leq C_{\textrm{HLS}} \|f\ind_{f>R}\|_{L^{p}}\|\phi\|^{2}_{L^{2r}}, \qquad \frac{1}{p}+\frac{2}{3}=2-\frac{1}{r}.$$
We choose then $p=2r$ which amounts to $p=\frac{9}{4}$ (notice that, for $d >3$, the choice of $p=2r$ would amount to $r=\frac{3d}{4d-4} \le1$ which is not admissible). %Then,
%\begin{equation*}
%J_{R} \leq  C\| f\ind_{f >R} \|_{L^{9/4}}\|\phi \|^{2}_{L^{9/4}}\,.
%\end{equation*}
Using then Lebesgue's interpolation,
\begin{align*}
\| f\ind_{f >R} \|_{L^{\frac{9}{4}}}\leq \| \langle\cdot\rangle^{2} f\ind_{f >R} \|^{\frac{1}{3}}_{L^{1}}\| \langle\cdot\rangle^{-1} f \|^{\frac{2}{3}}_{L^{6}} , \qquad 
 \|\phi \|_{L^{\frac{9}{4}}}\leq \| \langle \cdot \rangle^{2}\phi \|^{\frac{1}{3}}_{L^{1}}\| \langle \cdot \rangle^{-1} \phi \|^{\frac{2}{3}}_{L^{6}}\,.
\end{align*}
In addition, for $\theta_{s}=1-\frac{2}{s}$, for $s>2$, it holds that
$$\| \langle\cdot\rangle^{2} f\ind_{f >R} \|_{L^{1}}\leq \| \langle\cdot\rangle^{s} f\ind_{f >R} \|^{1-\theta_s}_{L^{1}}\| f\ind_{f >R} \|^{\theta_s}_{L^{1}}\leq \| \langle\cdot\rangle^{s} f \|^{1-\theta_s}_{L^{1}}\bigg(\frac{\| f\log f  \|_{L^{1}}}{\log(R)}\bigg)^{\theta_s}\,.$$
 Overall,
\begin{equation*}
J_{R} \leq  \frac{C}{\log(R)^{\frac{\theta_{s}}{3}}}\| \langle\cdot\rangle^{s} f \|^{\frac{1-\theta_s}{3}}_{L^{1}}\| f\log f  \|^{\frac{\theta_s}{3}}_{L^{1}}\| \langle \cdot \rangle^{2}\phi \|^{\frac{2}{3}}_{L^{1}}\| \langle \cdot \rangle^{-1} f \|^{\frac{2}{3}}_{L^{6}}\| \langle \cdot \rangle^{-1} \phi \|^{\frac{4}{3}}_{L^{6}}\,.
\end{equation*}
When considering a solution $f$ with moment $s>2$ and entropy bounded, this implies for $\phi=f$ that
\begin{equation*}
J_R \leq \frac{C(f_{\mathrm{in}},t)}{\log(R)^{\frac{\theta_{s}}{3}}}\| \langle \cdot \rangle^{-1} f \|^{2}_{L^{6}} \leq \frac{\tilde{C}(f_{\mathrm{in}},t)}{\log(R)^{\frac{\theta_{s}}{3}}}\| \nabla\big(\langle \cdot \rangle^{-1} f\big ) \|^{2}_{L^{2}}\end{equation*}
where we used Sobolev inequality and $\tilde{C}(f_{\rm in},t)$ depends on $\|f(t)\|_{L^{1}_{s}}$, $H(f_{\rm in})$ and energy. Then, for any $\e >0$, one obtains
$$J_{R} \leq  \varepsilon\| \nabla\big(\langle \cdot \rangle^{-1} f\big ) \|^{2}_{L^{2}}\,,\qquad \varepsilon>0\,,$$
by taking $R>1$ sufficiently large so that $R\geq \exp\left((\tilde{C}(f_{\rm in},t)/\varepsilon)^{\frac{3}{\theta_s}}\right)$.
\end{proof}

 \section{Appearances of $L^{p}$-norms in the soft potential case}\label{sec:DeG}
 
We investigate here the regularizing effects induced by the $\e$-Poincar\'e inequality and in particular the appearance of $L^{p}$-norms $(p >1)$. 
 \subsection{$L^{p}$-estimates for $1 < p < \infty$}\label{sec:Lp} More precisely, for any weak solution $f(t,\cdot)$ to \eqref{LFD1}, we introduce the notation
\begin{equation}\label{eq:notMsp}
\lM_{s,p}(t)=\int_{\R^{d}}f(t,v)^{p}\langle v\rangle^{s}\d v\,, \qquad \lD_{s,p}(t)=\int_{\R^{d}}\left|\nabla \left(\langle v\rangle^{\frac{s}{2}}f^{\frac{p}{2}}(t,v)\right)\right|^{2}\d v\end{equation}
where $p \in (1,\infty)$ is fixed in all the sequel and $s \in \R$. We also set
$$\lm_{s}(t)=\int_{\R^{d}}f(t,v)\langle v\rangle^{s}\d v, \qquad t \geq0, s \in \R.$$
One has the following evolution of $\lM_{s,p}(t)$. 
\begin{prop}
  \label{prop:L2} Let $-2\le\g<0$ and $T>0$. Let a nonnegative initial datum {$f_{\mathrm{in}} \in L^{1}_{2}(\R^{d}) \cap L\log L(\R^{d})$  satisfying  \eqref{eq:Mass} be given}. Let  $f(t,\cdot)$ be a weak-solution to \eqref{LFD}. Given $s \geq0$ and $1 < p<\infty$, there exists some positive constant $\bar{C}_{p,\g,s}(f_{\mathrm{in}})$ depending on  $p$, $\g$, $s$, $T$ and $f_{\rm in}$,  such that
\begin{equation}\label{final-L2}
 \dfrac{\d}{\d t}\lM_{s,p}(t) +  \frac{K_{0}(p-1)}{2p}\lD_{s+\g,p}(t) \leq \bar{C}_{p,\g,s}(f_{\mathrm{in}}) \lm_{\frac{s+\g}{p}}(t)^{p}\,
\end{equation} holds for any $t\in[0,T].$
\end{prop} 
\begin{proof} One checks easily that, for any $s \geq 0$,
 \begin{multline*}
\frac{1}{p}\frac{\d}{\d t} \int_{\R^{d}} f^p(t,v) \langle v\rangle^{s} \d v =\int_{\R^{d}}\langle v\rangle^{s}f^{p-1}(t,v)\nabla \cdot \left(\mathcal{A}[f]\nabla f-\bm{b}[f] f\right)\d v\\
=-(p-1)\int_{\R^{d}}\langle v\rangle^{s}f^{p-2}\mathcal{A}[f]\nabla f\cdot   \nabla f\d v + (p-1)\int_{\R^{d}}\langle v\rangle^{s}f^{p-1} \bm{b}[f]\cdot \nabla f\d v\\
-s\int_{\R^{d}}\langle v\rangle^{s-2}\,f^{p-1}\left(\mathcal{A}[f]\nabla f\right)\cdot v\d v +s \int_{\R^{d}}\langle v\rangle^{s-2}f^{p} \bm{b}[f]\cdot v\d v
\end{multline*}
where we used that $\nabla f^{p-1}=(p-1)f^{p-2}\nabla f,$ $\nabla \langle v\rangle^{s}=s\langle v\rangle^{s-2} v.$ Then, noticing that
$$\mathcal{A}[f]\nabla f^{\frac{p}{2}}\cdot \nabla f^{\frac{p}{2}}=\frac{p^{2}}{4}f^{p-2}\mathcal{A}[f]\nabla f\cdot \nabla f$$
and using the uniform ellipticity of the diffusion matrix $\mathcal{A}[f]$ (recall Proposition \ref{diffusion}), we deduce that
$$(p-1)\int_{\R^{d}}\langle v\rangle^{s}f^{p-2}\mathcal{A}[f]\nabla f\cdot   \nabla f\d v 
 \geq \frac{4K_{0}(p-1)}{p^{2}} \int_{\R^{d}}\langle v\rangle^{s+\g}\,\left|\grad f^{\frac{p}{2}}\right|^{2}\d v\,.$$
Moreover, writing
\begin{equation*}
\nabla \left(\langle v\rangle^{\frac{s+\g}{2}}\,f^{\frac{p}{2}}\right)=\langle v\rangle^{\frac{s+\g}{2}}\nabla f^{\frac{p}{2}} + \frac{s+\g}{2}v\,\langle v\rangle^{\frac{s+\g}{2}-2}f^{\frac{p}{2}}\,,
\end{equation*}
from which
\begin{equation}\label{eq:Gradient}
\langle v\rangle^{s+\g}\left|\nabla f^{\frac{p}{2}}\right|^{2} \geq \frac{1}{2}\left|\nabla \left(\langle v\rangle^{\frac{s+\g}{2}}f^{\frac{p}{2}}\right)\right|^{2} - \frac{1}{4}(s+\g)^{2} \langle v\rangle^{s+\g-2}f^{p}(v), \end{equation}
and 
\begin{multline*}
(p-1)\int_{\R^{d}}\langle v\rangle^{s}f^{p-2}\mathcal{A}[f]\nabla f\cdot   \nabla f\d v 
 \geq \frac{2K_{0}(p-1)}{p^{2}} \int_{\R^{d}}\left|\nabla \left(\langle v\rangle^{\frac{s+\g}{2}}f^{\frac{p}{2}}\right)\right|^{2}\d v\\
 -\frac{K_{0}(p-1)(s+\g)^{2}}{p^{2}}\int_{\R^{d}} \langle v\rangle^{s+\g-2}f^{p}(v)\d v\,.\end{multline*}
We also have
\begin{multline*}
 \int_{\R^{d}}\langle v\rangle^{s} f^{p-1}  \bm{b}[f] \cdot \grad f  \d v
 =  -\frac{1}{p} \int_{\R^{d}}  f^p \grad \cdot \Big(\bm{b}[f] \langle v\rangle^{s}\Big) \d v\\
=-\frac{s}{p}\int_{\R^{d}}\langle v\rangle^{s-2} f^p(t,v)  \bm{b}[f]\cdot v\,\d v
- \frac{1}{p}\int_{\R^{d}}\langle v\rangle^{s} f^p(t,v) \bm{c}_{\g}[f]\, \d v.
\end{multline*}
Therefore,   we get 
\begin{multline*}
 \dfrac{\d}{\d t}\lM_{s,p}(t) + \frac{2K_{0}(p-1)}{p} \lD_{s+\g,p}(t) \\
 \leq  -  {(p-1)} \int_{\R^{d}}\langle v\rangle^{s} f^p(t,v)\,\bm{c}_{\g}[f]\d v
+s \int_{\R^{d}}\langle v\rangle^{s-2} f^p(t,v)\,\left(\bm{b}[f]\cdot v\right)\d v  \\
+ \frac{K_{0}(p-1)(s+\g)^{2}}{p}\int_{\R^{d}} \langle v\rangle^{s+\g-2}f^{p}(v)\d v
-sp\int_{\R^{d}}\langle v\rangle^{s-2}f^{p-1} \left(\mathcal{A}[f]\nabla f\cdot v \right)\d v.
\end{multline*}
Let us investigate more carefully the last term. Integration by parts shows that \begin{align*}
-sp\int_{\R^{d}}\langle v\rangle^{s-2}f^{p-1} \left(\mathcal{A}[f]\nabla f\cdot v \right)\d v &= -s\int_{\R^{d}}\nabla f^{p} \cdot \Big(\mathcal{A}[f]
\, \langle v\rangle^{s-2}v \Big) \,\d v\\
&= s\int_{\R^{d}} f^{p} \; \nabla\cdot \Big(\mathcal{A}[f]\,\langle v\rangle^{s-2} v \Big) \,\d v\,.
\end{align*}
Using the product rule 
$$\nabla\cdot \Big(\mathcal{A}[f]\,\langle v\rangle^{s-2} v \Big)=\langle v\rangle^{s-2}\,{\bm{b}}[f]\cdot v\, + \mathrm{Trace}\left(\mathcal{A}[f] \cdot \,
D_{v}\left(\langle v\rangle^{s-2} v\right)\right), $$ 
where $D_{v}\big( \langle v\rangle^{s-2} v\big)$ is the matrix with entries $\partial_{v_{i}}\big( \langle v\rangle^{s-2} v_{j}\big)$, $i,j=1,\ldots,d$, or more compactly, 
$$D_{v}\big( \langle v\rangle^{s-2} v\big)=\langle v\rangle^{s-4}\bm{A}(v)\,,$$
where $\bm{A}(v)=\langle v\rangle^{2}\mathbf{Id}+(s-2)\,v\otimes v$, $v \in \R^{d}$. We obtain

\begin{multline}\label{eq:dtMs}
 \dfrac{\d}{\d t}\lM_{s,p}(t) + \frac{2K_{0}(p-1)}{p}\lD_{s+\g,p}(t) \leq -  {(p-1)}\int_{\R^{d}}\langle v\rangle^{s} \bm{c}_{\g}[f(t)](v)\,f^p(t,v) \d v\\
+s \int_{\R^{d}}\langle v\rangle^{s-2}\,f^p(t,v) \left(\bm{b}[f]\cdot v\right)\d v  +  \frac{K_{0}(p-1)(s+\g)^{2}}{p}\int_{\R^{d}}\langle v\rangle^{s+\g-2}f^{p}(t,v)\d v
\\
+s\int_{\R^{d}}\langle v\rangle^{s-4}f^{p}\,\, \mathrm{Trace}\left(\mathcal{A}[f]\cdot \bm{A}(v)\right)\d v\,.
\end{multline}
We denote by $I_{1},I_{2},I_{3},I_{4}$ the various terms on the right-hand-side of \eqref{eq:dtMs}
and we control each term separately. Applying Theorem \ref{cor:MAIN} with   $\phi =\langle \cdot \rangle^{\frac{s}{2}}f^{\frac{p}{2}}(t,v)$, we deduce that, for any $\delta \in (0,1)$,
$$|I_{1}| \leq  {(p-1)}\left(\delta\,\lD_{s+\g,p}(t)+C_{1}\lM_{s+\g,p}(t)\right)\,,$$
where $C_{1}$ is defined in Theorem \ref{cor:MAIN}. For the term $I_{2}$,  it  holds that
$$|I_{2}| \leq s\int_{\R^{d}}\langle v\rangle^{s-1}f^{p}(t,v)\,|\bm{b}[f(t)](v)|\d v \leq s(d-1)\,\int_{\R^{2d}}\langle v\rangle^{s-1}f^{p}(t,v)|v-\vet|^{\g+1}f(t,\vet)\d \vet\d v\,.$$
Therefore, if $\g+1 <0$, applying Theorem \ref{cor:MAIN} with  {$\bm{c}_{\g+1}[f(t)]$ instead of $\bm{c}_{\g}[f(t)]$} and $\phi^{2}=\langle \cdot\rangle^{s-1}f^{p}(t)$, we get
$$|I_{2}| \leq  {\frac{s}{d+\g+1}}\left(\delta\,\lD_{s+\g,p}(t) + C_{1}\,\lM_{s+\g,p}(t)\right),$$
 whereas, if $\g+1 >0$, one has obviously $|I_{2}| \leq s(d-1)\|\langle \cdot\rangle^{\gamma+1}f(t)\|_{L^{1}}\lM_{s+\g,p}(t).$
In both cases, for any  {$\delta >0$}, 
$$|I_{2}| \leq  {s \,C(d,\g)}\left(\delta\,\lD_{s+\g,p}(t) + C_{1}\,\lM_{s+\g,p}(t)\right).$$
Clearly,
$$|I_{3}| \leq  \frac{K_{0}(p-1)(s+\g)^{2}}{p}\lM_{s+\g,p}(t).$$
For the term $I_{4}$, one checks easily that, for any $i,j \in \{1,\ldots,d\}$,
$$\left|\mathcal{A}_{i,j}[f]\right| \leq |\cdot|^{\g+2}\ast f, \qquad \left|\bm{A}_{i,j}(v)\right| \leq s\langle v\rangle^{2}, $$
and 
$$|I_{4}| \leq d^{2}s^{2}\int_{\R^{2d}}\langle v\rangle^{s-2}f^{p}(t,v)|v-\vet|^{\g+2}f(t,\vet)\d v\d\vet.$$
One has, since $\g+2 \geq0$,
$$|I_{4}| \leq s^{2}\,d^{2}\|\langle \cdot \rangle^{\g+2}f(t)\|_{L^{1}}\,\|\langle\cdot\rangle^{\g+s}f^{p}(t)\|_{L^{1}}=s^{2}\,d^{2}\|\langle \cdot \rangle^{\g+2}f(t)\|_{L^{1}}\,\lM_{s+\g,p}(t).$$
Overall, recalling mass and energy conservation to estimate all the weighted $L^{1}$-terms, one sees that, for any ${\delta \in (0,1)}$,  there is some positive constant $C_{p,\delta, \g}(f_{\mathrm{in}})$ depending on  $f_{\mathrm{in}}$, $\g$, $\delta$  and  $p$ such that
\begin{equation}\label{eq:Ms+g0}\begin{split} \dfrac{\d}{\d t}\lM_{s,p}(t) &+ \frac{2K_{0}(p-1)}{p}\lD_{s+\g,p}(t) \\
&\leq C_{p,\delta,\g}(f_{\mathrm{in}})(1+s^{2})\lM_{s+\g,p}(t) +   {(sC(d,\g)+p)}\delta\lD_{s+\g,p}(t)\,.\end{split}\end{equation}
Picking then $\delta \in (0,1)$ such that $ {(sC(d,\g)+p)}\delta \leq \frac{K_{0}(p-1)}{p}$
 one deduces that
\begin{equation}\label{eq:Ms+gs0}
\dfrac{\d}{\d t}\lM_{s,p}(t) + \frac{K_{0}}{q}\lD_{s+\g,p}(t)
\leq \widetilde{C}_{p,\g,s}(f_{\mathrm{in}})\lM_{s+\g,p}(t), \qquad s \geq0,\quad \frac{1}{p}+\frac{1}{q}=1\,,\end{equation}
for some positive constant $\widetilde{C}_{p,\g,s}(f_{\rm in})$ depending only on $p$, $\g$, $s$ and $f_{\rm in}$. We now estimate the right-hand-side $\lM_{s+\g,p}(t)$ by suitable interpolation between $L^{1}$-moment of $f$ and $\lD_{s+\g,p}(t).$ First, observe that Sobolev'inequality reads
\begin{equation}\label{eq:sobo-p}
\left\|\langle \cdot\rangle^{\frac{s+\g}{p}}f\right\|_{L^{q}}^{p}=\left\|\langle \cdot\rangle^{\frac{s+\g}{2}}f^{\frac{p}{2}}\right\|_{L^{\frac{2d}{d-2}}}^{2} \leq C_{\mathrm{Sob}}\lD_{s+\g,p}(t), \qquad q=\frac{p\,d}{d-2}.\end{equation}
Now, recall the interpolation inequality
\begin{equation}\label{int-ineq}
\|\langle \cdot \rangle^{a}g\|_{L^{r}} \leq \|\langle \cdot \rangle^{a_{1}}g\|_{L^{r_{1}}}^{\theta}\,\|\langle \cdot \rangle^{a_{2}}g\|_{L^{r_{2}}}^{1-\theta}\,,
\end{equation}
with 
$$\frac{1}{r}=\frac{\theta}{r_{1}}+\frac{1-\theta}{r_{2}}, \quad a=\theta\,a_{1}+(1-\theta)a_{2},  \quad \theta \in (0,1).$$
Applying this with $r=p$, $r_{1}=1,$ $r_{2}=q=\frac{pd}{d-2} >p$  (so that $\theta=\frac{2}{d(p-1)+2}$)
and $a=a_{1}=a_{2}=\frac{s+\g}{p}$ we deduce
\begin{equation}\begin{split}\label{eq:ms+gMs+g}
\lM_{s+\g,p}(t) &\leq \lm_{\frac{s+\g}{p}}(t)^{\frac{2p}{d(p-1)+2}}\,\left\|\langle \cdot\rangle^{\frac{s+\g}{p}}f\right\|_{L^{q}}^{\frac{d(p-1)p}{d(p-1)+2}}\\
&\leq C_{\mathrm{Sob}}^{\frac{d(p-1)}{d(p-1)+2}}\lm_{\frac{s+\g}{p}}(t)^{\frac{2p}{d(p-1)+2}}\,\lD_{s+\g,p}(t)^{\frac{d(p-1)}{d(p-1)+2}}\end{split}\end{equation}
where we used \eqref{eq:sobo-p} in that last estimate. Now, Young's inequality implies that there is $C   >0$ such that, for any $\alpha >0$,  
\begin{equation*}\label{eq:MsgYo}
\,\lM_{s+\g,p}(t) \leq C {\alpha}^{-\frac{d(p-1)}{2}} \lm_{\frac{s+\g}{p}}(t)^{p}+\alpha\lD_{s+\g,p}(t).\end{equation*}
Choosing now $\alpha >0$ such that $\widetilde{C}_{p,\g,s}(f_{\rm in})\alpha= \frac{K_{0}(p-1)}{2p}$, we end up with
$$\dfrac{\d}{\d t}\lM_{s,p}(t) + \frac{K_{0}}{2q}\lD_{s+\g,p}(t) \leq \bar{C}_{p,\g,s}(f_{\mathrm{in}})\lm_{\frac{s+\g}{p}}(t)^{p}\,.$$
This shows \eqref{final-L2}.
\end{proof}

The above result provides the following \emph{appearance} for $\lM_{s,p}(\cdot)$. 

\begin{theo}\label{theo:boundedL2}  Let $-2\le \g <0$, $s\ge0$, $p>1$ and $T>0$. Let a nonnegative initial datum $f_{\mathrm{in}} \in L^{1}_{2}(\R^{d}) \cap L\log L(\R^{d})$ satisfying  \eqref{eq:Mass} be given. Assume additionally that
{  \begin{equation}\label{eq:extra_assumption}
    f_{\mathrm in} \in L^{1}_{\mu(s,p)}(\R^d), \qquad \mu(s,p):= \frac{2s-\g d(p-1)}{2p}\,,
    \end{equation}}
and  let  $f(t,\cdot)$ be a weak-solution to \eqref{LFD}. Then, there exists a constant $c_{s,\g,p}(f_{\mathrm{in}})$  such that
\begin{equation}\label{appearLMs}
\lM_{s,p}(t) \leq c_{s,\g,p}(f_{\rm in})\max\left(1,t^{-\frac{d(p-1)}{2}}\right),\qquad  {t\in(0,T)}\,.
\end{equation} \end{theo}
\begin{proof} Let us pick $s \geq 0$. Recall estimate \eqref{final-L2} 
\begin{equation*}
\dfrac{\d}{\d t}\lM_{s,p}(t) + {\frac{K_{0}(p-1)}{2p}}\lD_{s+\g,p}(t) \leq  \bar{C}_{p,\g,s}(f_{\mathrm{in}}) \lm_{\frac{s+\g}{p}}(t)^{p}\,,\qquad t>0\,.
\end{equation*} We resort again to the interpolation inequality \eqref{int-ineq} to estimate now $\lM_{s,p}(t)$ in terms of moments of $f$ and $\lD_{s+\g,p}(t)$. Namely, applying \eqref{int-ineq} with $r=p$, $r_{1}=1,$ $r_{2}=q=\frac{pd}{d-2}$, i.e. $\theta=\frac{2}{d(p-1)+2}$,  { $a=\frac{s}{p}$} and $a_{2}=\frac{s+\g}{p}$, we see that 
$$a_{1}=\frac{2s-\g d(p-1)}{2p}$$
and, as in \eqref{eq:ms+gMs+g}, we obtain 
$$\lM_{s,p}(t) \leq C_{\mathrm{Sob}}^{\frac{d(p-1)}{d(p-1)+2}}\lm_{a_{1}}(t)^{\frac{2p}{d(p-1)+2}}\,\lD_{s+\g,p}(t)^{\frac{d(p-1)}{d(p-1)+2}}.$$
Thus,
\begin{equation}\label{eq:compare}
\lD_{s+\g,p}(t) \geq C_{d,p}\,\lm_{\frac{2s-\g d(p-1)}{2p}}(t)^{-\frac{2p}{d(p-1)}}\,\lM_{s,p}(t)^{\frac{d(p-1)+2}{d(p-1)}}\end{equation}
which, with \eqref{final-L2}, results in the inequality valid for any $t>0$
\begin{equation}\label{eq:MspK}
\dfrac{\d}{\d t}\lM_{s,p}(t) + \bm{C}_{d,p}K_{0}\lm_{\frac{2s-\g d(p-1)}{2p}}(t)^{-\frac{2p}{d(p-1)}}\,\lM_{s,p}(t)^{\frac{d(p-1)+2}{d(p-1)}} \leq  \bar{C}_{p,\g,s}(f_{\mathrm{in}})\lm_{\frac{s+\g}{p}}(t)^{p}\,.
\end{equation}
Now, let us note that $\frac{2s-\g d(p-1)}{2p}> \frac{s+\g}{p}$. If $\frac{2s-\g d(p-1)}{2p}\le2$,  we clearly have 
$$\sup_{t\in[0,T]}\max\left(\lm_{\frac{2s-\g d(p-1)}{2p}}(t),\lm_{\frac{s+\g}{p}}(t)\right) \le\|f_{\mathrm{in}}\|_{L^1_2}.$$
Otherwise, we deduce from the assumption \eqref{eq:extra_assumption} and Theorem \ref{theo:prop_mom} that there exists $C_{s,\g}(p)$ depending on $s$, $\g$, $p$ and $f_{\mathrm{in}}$ such that
$$\sup_{t\in[0,T]}\max\left(\lm_{\frac{2s-\g d(p-1)}{2p}}(t),\lm_{\frac{s+\g}{p}}(t)\right)\le C_{s,\g}(p).$$
Therefore, there exist $\bm{a}_{s}(f_{\rm in})$ and $\bm{k}_{s}(f_{\rm in})$ depending on $\g,p,s,T$ and $f_{\mathrm{in}}$ such that 
\begin{equation}\label{eq:xtyt}
\dfrac{\d}{\d t}\lM_{s,p}(t) + \bm{a}_{s}(f_{\rm in})\lM_{s,p}(t)^{\frac{d(p-1)+2}{d(p-1)}}  \leq  \bm{k}_s(f_{\rm in})\,,\qquad t \in[0,T]\,.
\end{equation}
The conclusion then follows by a comparison argument. Namely, one shows that \eqref{appearLMs} holds true with
$$c_{s,\g,p}(f_{\mathrm{in}})= \max\left\{\left(\frac{d(p-1)}{\bm{a}_{s}(f_{\rm in})}\right)^{\frac{d(p-1)}{2}}\,,\,{2^{\frac{d(p-1)}{d(p-1)+2}}} \left(\frac{\bm{k}_{s}(f_{\rm in})}{\bm{a}_{s}(f_{\rm in})}\right)^{\frac{d(p-1)}{d(p-1)+2}} \right\}$$
by simply resuming the arguments of \cite[proof of Proposition 3.12]{ABDL1}.
\end{proof} 

  \subsection{De Giorgi's approach to pointwise bounds}\label{sec:level}

We now show how the above appearance of $L^{p}$-norms implies the appearance of $L^{\infty}$ bounds. For parabolic equations, the passage from $L^{2}$ to $L^{\infty}$ is made through the so-called De Giorgi-Moser iteration procedure \cite{degiorgi} and such an approach has been introduced in \cite{ricardo} for spatially homogenous kinetic equations. 
We introduce, as in \cite{ricardo}, for any \emph{fixed} $\l >0$,
$$f_{\l}(t,v) :=(f(t,v)-\l), \qquad f_{\l}^{+}(t,v) :=f_{\l}(t,v)\ind_{\{f\geq\l\}}.$$
To prove an $L^{\infty}$ bound for $f(t,v)$, one looks for an $L^{2}$-bound for $f_{\l}$. We start with the following estimate.  
\begin{prop}\label{lem:fl+}  Let $-2\le \g<0$ and $T>0$. Let a nonnegative initial datum $f_{\mathrm{in}}\in L^1_2(\R^d)\cap L\log L(\R^d)$ satisfying \eqref{eq:Mass} be given. Let  $f(t,\cdot)$ be a weak-solution to \eqref{LFD}.
There exist $c_{0},\kappa_{0} >0$ depending only on $\g$, $T$ and $f_{\rm in}$  such that, for any $\l >0$, any $t\in[0,T]$, 
\begin{equation}\label{eq:13Ric}\begin{split}
\frac{1}{2}\frac{\d}{\d t}\|f_{\l}^{+}(t)\|_{L^{2}}^{2} &+ c_{0}\int_{\R^{d}}\left|\nabla \left(\langle v \rangle^{\frac{\g}{2}}f_{\l}^{+}(t,v)\right)\right|^{2}\d v \\
&\leq \kappa_{0} \|\langle \cdot \rangle^{\frac{\g}{2}}f_{\l}^{+}(t)\|_{L^{2}}^{2}-\l \int_{\R^{d}}\bm{c}_{\g}[f](t,v)\,f_{\l}^{+}(t,v)\d v.\end{split}\end{equation}
\end{prop}

\begin{proof} Given $\l >0$, one has
$\partial_{t}\left(f_{\l}^{+}\right)^{2}=2f_{\l}^{+}\partial_{t}f_{\l}^{+}=2f_{\l}^{+}\partial_{t}f$
and $\nabla f_{\l}^{+}=\ind_{\{f\geq\l\}}\nabla f$, so that
\begin{multline*}
\frac{1}{2}\frac{\d}{\d t}\|f_{\l}^{+}(t)\|_{L^2}^{2}=-\int_{\R^{d}}\mathcal{A}\nabla f\cdot \nabla f_{\l}^{+} \d v
+\int_{\R^{d}}f \bm{b}[f]\cdot \nabla f_{\l}^{+}\d v\\
=-\int_{\R^{d}}\mathcal{A}\nabla f^{+}_{\l}\cdot \nabla f_{\l}^{+}\d v
+\int_{\R^{d}}f \bm{b}[f]\cdot \nabla f_{\l}^{+}\d v.\end{multline*}
Now, one easily checks that
\begin{equation*} 
f \nabla f^{+}_{\l}=\tfrac{1}{2} \nabla (f^{+}_{\l})^{2}+ \l \nabla f^{+}_{\l}.
 \end{equation*}
Therefore
\begin{equation*}
\frac{1}{2}\frac{\d}{\d t}\|f_{\l}^{+}(t)\|_{L^2}^{2}+\int_{\R^{d}}\mathcal{A}\nabla f^{+}_{\l}\cdot \nabla f^{+}_{\l}\d v
= -\tfrac{1}{2} \int_{\R^{d}}\bm{c}_{\g}[f](f^{+}_{\l})^{2}\d v - \l \int_{\R^{d}}\bm{c}_{\g}[f] f^{+}_{\l}\d v .\end{equation*}
{It then follows from Proposition \ref{diffusion} that
\begin{equation*}\begin{split}
  \int_{\R^{d}}\mathcal{A}\nabla f^{+}_{\l}\cdot \nabla f^{+}_{\l}\d v & \ge  
  K_0 \int_{\R^{d}} \langle v\rangle^\g |\nabla f^{+}_{\l}(t,v)|^2 \d v  \\
& \ge   \frac{K_0}{2} \int_{\R^{d}} \left|\nabla \left( \langle v\rangle^{\frac{\g}{2}}f^{+}_{\l}(t,v)\right)\right|^2 \d v  - K_0 \int_{\R^{d}}  \left(f^{+}_{\l}(t,v)\right)^2 \left|\nabla \left( \langle v\rangle^{\frac{\g}{2}}\right)\right|^2 \d v\\
& \ge   \frac{K_0}{2} \int_{\R^{d}} \left|\nabla \left( \langle v\rangle^{\frac{\g}{2}}f^{+}_{\l}(t,v)\right)\right|^2 \d v  - \tilde{C}K_0  \|\langle \cdot \rangle^{\frac{\g}{2}}f_{\l}^{+}(t)\|_{L^{2}}^{2}\,.
\end{split}\end{equation*}
We finally deduce from \eqref{eq:poin0}  with $\phi=f_{\l}^{+}$ and $\varepsilon =\frac{K_0}{2}$ that \eqref{eq:13Ric} holds  with $c_0:=\frac{K_0}{4}$ and $\kappa_{0}:=\frac{1}{2}C + \tilde{C} K_0$, where $C$ is defined in Theorem \ref{cor:MAIN}.}
\end{proof} 

Inspired by De Giorgi's iteration method introduced for elliptic equations, the crucial point in the level set approach of \cite{ricardo} is to compare some suitable energy functional associated to $f_{\l}^{+}$ with the same energy functional at some different level $f_{k}^{+}.$ The key observation being that, if 
$0 \leq k <\ell$, then
\begin{equation*}
0 \leq f^{+}_{\l}\leq f^{+}_{k},\quad\text{and}\quad \mathbf{1}_{\{f_{\l} \geq 0\}} \leq \frac{f^{+}_{k}}{\ell - k}\,.
\end{equation*}
This implies in particular that (see \cite{ricardo,ABDL1} for details)
\begin{equation}\label{eq:alphaf}
f_{\l}^{+} \leq \left(\ell-k\right)^{-\alpha}\,\left(f^{+}_{k}\right)^{1+\alpha} \qquad \forall \,\alpha \geq 0, \qquad 0 \leq k < \l.\end{equation}
On this basis, we need the following interpolation inequalities which are independent of the equation \eqref{LFD}.
\begin{lem}\label{lem:changelevel}  There exists $C >0$  such that, for any $0 \leq k < \l$, one has (for all nonnegative $f$ for which the terms are defined),  
\begin{equation}\label{eq:flL2}
\|\langle \cdot \rangle^{\frac{\g}{2}}f_{\l}^{+}\|_{L^{2}}^{2} \leq C\,(\l-k)^{- {\frac{4}{d}}}\left\|\nabla \left(\langle \cdot \rangle^{\frac{\g}{2}}f_{k}^{+}\right)\right\|_{L^{2}}^{2}\,\left\|f^{+}_{k}\right\|_{L^{2}}^{ {\frac{4}{d}}}.\end{equation}
For $p \in \left[1, {\frac{d}{d-2}}\right)$, there is $C_{p} >0$ such that, for any $0 \leq k < \l$ (and all nonnegative $f$ for which the terms are defined), 
\begin{equation}\label{eq:flLp}
\|\langle \cdot \rangle^{\g}f_{\l}^{+}\|_{L^{p}} \leq C_{p}(\l-k)^{-( {\frac{2}{p}-\frac{d-4}{d}})}\,\left\|\nabla\left(\langle \cdot \rangle^{\frac{\g}{2}}f_{k}^{+}\right)\right\|_{L^{2}}^{2}\,\|f_{k}^{+}\|_{L^{2}}^{ {\frac{2}{p}+\frac{4-2d}{d}}}.\end{equation}
Moreover, for any  $q \in  \left( {\frac{2d+2}{d}} , {\frac{2d+4}{d}}\right)$, there is $c_{q} >0$ such that, for any $0 \leq k < \l$ (and all nonnegative $f$ for which the terms are defined), 
\begin{equation}\label{eq:flLq}
\|f_{\l}^{+}\|_{L^{2}}^{2} \leq \frac{c_{q}}{(\l-k)^{q-2}}\,\left\|\langle \cdot \rangle^{s}f_{k}^{+}\right\|_{L^{1}}^{ {\frac{2d+4}{d}}-q}\,\|f^{+}_{k}\|_{L^{2}}^{2(q- {\frac{2d+2}{d}})}\,\left\|\nabla \left(\langle \cdot \rangle^{\frac{\g}{2}}\,f^{+}_{k}\right)\right\|_{L^{2}}^{2},
\end{equation}
with $s=- {\frac{\g d}{2d+4-q d}} > - {\frac{d}{2}\g}$. Finally, for $s>2$, there exists $c_{s}>0$ such that, for any $0 \leq k < \l$ (and all nonnegative $f$ for which the terms are defined),
\begin{equation}\label{eq:flLd/(d-1)}
\|f_{\l}^{+}\|_{L^{\frac{d}{d-1}}}^{2} \leq c_{s} (\l-k)^{-2\frac{s-1}{s}}\,\left\|\langle \cdot \rangle^{s}f_{k}^{+}\right\|_{L^{1}}^{\frac{2}{s}}\,\|f^{+}_{k}\|_{L^{2}}^{\frac{2(s-2)}{s}}\,\left\|\nabla \left(\langle \cdot \rangle^{-1}\,f^{+}_{k}\right)\right\|_{L^{2}}^{2}.
\end{equation}
\end{lem}
 \begin{proof} The proof of this result follows the paths of the analogous one \cite[Lemma 4.2]{ABDL1}. However, since the latter was restricted to the case $d=3$ and $-2 < \g <0$, we give a full proof here to agree on the various exponents. The basic argument is the interpolation inequality \eqref{int-ineq}.
Moreover, for the special case $r_{1}= {\frac{2d}{d-2}}$, $r_{2}=2$, $a_{1}=\frac{\g}{2}$ and $r \in (2,r_{1})$, thanks to Sobolev embedding, the identity will become
\begin{equation}\begin{split}\label{eq:interpol}
\|\langle \cdot \rangle^{a}g\|_{L^{r}} &\leq C_{\theta}\,\left\|\nabla\left(\langle \cdot \rangle^{\frac{\g}{2}}g\right)\right\|_{L^{2}}^{\theta}\,\|\langle \cdot \rangle^{a_{2}}g\|_{L^{2}}^{1-\theta},\\
\frac{1}{r}&= {\frac{d-2\theta}{2d}}, \qquad a=\theta\,\frac{\g}{2}+(1-\theta)a_{2}, \qquad \theta \in (0,1),\qquad r \in (2,r_{1}).\end{split}\end{equation}
With these tools at hands, one has for $0 \leq k < \l$ and $r >2$, writing $r=2+2\alpha$ with \eqref{eq:alphaf},
\begin{multline*}
\|\langle \cdot \rangle^{\frac{\g}{2}}f_{\l}^{+}\|_{L^{2}}^{2}=\int_{\R^{d}}\langle v\rangle^{\g}(f^{+}_{\l}(t,v))^{2}\d v
\\\leq (\l-k)^{-2\alpha}\int_{\R^{d}}\langle v\rangle^{\g}(f^{+}_{k}(t,v))^{2+2\alpha}\d v
=(\l-k)^{-(r-2)}\left\|\langle \cdot \rangle^{\frac{\g}{r}}f^{+}_{k}(t)\right\|_{L^{r}}^{r}\,,
\end{multline*}
so that \eqref{eq:interpol} gives, with $a=\frac{\g}{r}$, 
$$\|\langle \cdot \rangle^{\frac{\g}{2}}f_{\l}^{+}\|_{L^{2}}^{2} \leq C(\l-k)^{-(r-2)}\left\|\nabla \left(\langle \cdot \rangle^{\frac{\g}{2}}f_{k}^{+}(t)\right)\right\|_{L^{2}}^{r\theta}\,\left\|\langle \cdot \rangle^{a_{2}} f^{+}_{k}\right\|_{L^{2}}^{r(1-\theta)}, $$
with $\theta= {\frac{d(r-2)}{2r}}$ and $a_{2}= {\frac{\g}{2}\frac{4+2d-dr}{2d+2r-dr}}$. One picks then $r={\frac{4+2d}{d}}$ so that $a_{2}=0$ and $r\theta=2$, to obtain \eqref{eq:flL2}. One proceeds in the same way to estimate $\|\langle \cdot\rangle^{\g}f_{\l}^{+}\|_{L^{p}}^{p}$. Namely, for $r >p$,
$$
\|\langle \cdot \rangle^{\g}f_{\l}^{+}\|_{L^{p}}^{p} \leq (\l-k)^{-(r-p)}\left\|\langle \cdot\rangle^{\frac{\g\,p}{r}}f^{+}_{k}\right\|_{L^{r}}^{r}
$$
and, with $r >2p$, imposing in \eqref{eq:interpol}  $a_{2}=0$ and $a=\frac{\g\,p}{r}$, we get $\theta=\frac{2p}{r}$ and 
$$\|\langle \cdot \rangle^{\g}f_{\l}^{+}\|_{L^{p}}^{p} \leq C(\l-k)^{-(r-p)}\left\|\nabla \left(\langle \cdot \rangle^{\frac{\g}{2}}f_{k}^{+}\right)\right\|_{L^{2}}^{2p}\,\left\|f_{k}^{+}\right\|_{L^{2}}^{r-2p}, $$ 
which gives \eqref{eq:flLp} when $r= {2+\frac{4p}{d}}$ (notice that $r >2p$ since $p < {\frac{d}{d-2}}$).\\

\noindent
Let us now prove \eqref{eq:flLq}. Let us consider first $q >2$ and use \eqref{int-ineq}. One has
\begin{equation}\label{interp}
  \|g\|_{L^{q}}\leq \|\langle \cdot \rangle^{s}\,g\|_{L^{1}}^{\theta_{1}}\,\|g\|_{L^{2}}^{\theta_{2}}\,\|\langle \cdot \rangle^{\frac{\g}{2}}g\|_{L^{ {\frac{2d}{d-2}}}}^{\theta_{3}},
  \end{equation}
with $\theta_{i} \geq 0$ $(i=1,2,3)$ such that
$$\theta_{1}+\theta_{2}+\theta_{3}=1, \qquad s\,\theta_{1}+0\cdot \theta_{2}+\frac{\g}{2}\theta_{3}=0, \qquad \frac{\theta_{1}}{1}+\frac{\theta_{2}}{2}+\frac{\theta_{3}(d-2)}{2d}=\frac{1}{q}.$$
Imposing $q\theta_{3}=2$, this easily yields
$$q\theta_{1}= {\frac{2d+4}{d}-q}, \qquad q\theta_{2}=2\left(q- {\frac{2d+2}{d}}\right), \qquad s=- {\frac{\g d}{2d+4-q d}}, \qquad q \in \left( {\frac{2d+2}{d}}, {\frac{2d+4}{d}}\right).$$
Using Sobolev inequality, this gives, for any $q \in \left( {\frac{2d+2}{d}}, {\frac{2d+4}{d}}\right)$,
 the existence of a positive constant $C >0$ such that
$$\|g\|_{L^{q}}^{q} \leq C\,\|\langle \cdot \rangle^{s}g\|_{L^{1}}^{ {\frac{2d+4}{d}-q}}\,\|g\|_{L^{2}}^{2\left(q- {\frac{2d+2}{d}}\right)}\,\left\|\nabla \left(\langle \cdot \rangle^{\frac{\g}{2}}\,g\right)\right\|_{L^{2}}^{2}, \qquad s=-{\frac{\g d}{2d+4-q d}}.$$
Using then \eqref{eq:alphaf}, for any $q >2$, one has $\|f_{\l}^{+}\|_{L^{2}}^{2} \leq (\l-k)^{2-q}\,\|f_{k}^{+}\|_{L^{q}}^{q}$ for $0 \leq k < \l$, and the above inequality gives the result.\\

It remains to prove \eqref{eq:flLd/(d-1)}. We deduce from \eqref{eq:alphaf} with $\alpha\ge0$ that, 
$$\|f_{\l}^{+}\|_{L^{\frac{d}{d-1}}}^2 \le (\l-k)^{-2\alpha} \|f_{k}^{+}\|_{L^{\frac{d(1+\alpha)}{d-1}}}^{2(1+\alpha)}.$$
Then, imposing $\alpha\in\left(0,\frac12\right)$, one may apply \eqref{interp} with $\gamma=-2$, $q=\frac{d(1+\alpha)}{d-1}$, $\theta_1=\frac{1-\alpha}{1+\alpha}$, $\theta_2=\frac{2\alpha-1}{1+\alpha}$, $\theta_3=\frac{1}{1+\alpha}$, $s=\frac{1}{1-\alpha}$ and obtain
$$\|f_{\l}^{+}\|_{L^{\frac{d}{d-1}}}^2 \le(\l-k)^{-2\alpha} \,\left\|\langle \cdot \rangle^{s}f_{k}^{+}\right\|_{L^{1}}^{2(1-\alpha)}\,\|f^{+}_{k}\|_{L^{2}}^{2(2\alpha-1)}\,\left\|\langle \cdot \rangle^{-1}\,f^{+}_{k}\right\|_{L^{\frac{2d}{d-2}}}^{2}.$$
Since $\alpha=1-\frac{1}{s}$,  \eqref{eq:flLd/(d-1)} follows directly from the above inequality and the Poincar\'e inequality.
\end{proof}

Let us now introduce, for any measurable $f:=f(t,v) \ge 0$ and  $\ell \geq 0$, the energy functional
$$\mathscr{E}_{\ell}(T_{1},T_{2})=\sup_{t \in [T_{1},T_{2})}\left(\frac{1}{2} \left\|f_{\l}^{+}(t)\right\|_{L^{2}}^{2} + c_{0}\int_{T_{1}}^{t}\left\|\nabla \left(\langle \cdot \rangle^{\frac{\g}{2}}\,f^{+}_{\ell}(\tau)\right)\right\|_{L^{2}}^{2}\d \tau\right), \qquad 0 \leq T_{1} \leq T_{2}$$
where $c_{0}$ is defined here above. From now on, we distinguish between the two cases $-2 < \g <0$ and $\g=-2$ since they lead to different kinds of estimates for the integral involving $\bm{c}_{\g}$ in \eqref{eq:13Ric}.

\subsubsection{Case $-2 < \g <0$}

We then have the fundamental result for the implementation of the level set analysis.
%\textcolor{red}{Maybe we can simplify to the worst homogeneity below}
\begin{prop}\label{main-energy-functional} Let $-2<\g<0$ and $T>0$. Let a nonnegative initial datum $f_{\mathrm{in}} \in L^1_2(\R^d)\cap L\log L(\R^d)$ satisfying \eqref{eq:Mass} be given. Let  $f(t,\cdot)$ be a weak-solution to \eqref{LFD}. Then, for any 
$$p_{\g} \in \left(\frac{d}{d+\g},3\right), \qquad  s > {\frac{d}{2}}|\g|,$$ there exist some positive constants $ C_{1},C_{2}$ depending only on $s$, $\|f_{\rm in}\|_{L^1_2}$ and $H(f_{\rm in})$ such that, for any times $0 \leq T_{1} < T_{2} \leq T$ and $0 \leq k < \ell$, 
\begin{multline}\label{eq:ElT2T3}
\mathscr{E}_{\l}(T_{2},T) \leq \frac{{C}_{2}}{T_{2}-T_{1}}(\l-k)^{{-\frac{4s+d\g}{ds}}}\left[\sup_{\tau \in [T_{1},T]}\lm_{s}(f(\tau))\right]^{{\frac{|\g|}{s}}}\,\mathscr{E}_{k}(T_{1},T)^{{\frac{(d+2)s+d\g}{ds}}}
\\
+{C}_{1}\mathscr{E}_{k}(T_{1},T)^{\frac{1}{p_{\g}}+\frac{2}{d}}\left(\l-k\right)^{-\left(\frac{2}{p_{\g}}-\frac{d-4}{d}\right)} {\sup_{\tau \in [T_{1},T]}\lm_{|\g|}(f(\tau))}
\\
\times \left(\l+\left[{(\l-k)^{\frac{2}{p_{\g}}-1}+\l(\l-k)^{\frac{2}{p_{\g}}-2}}\right]\mathscr{E}_{k}(T_{1},T)^{1-\frac{1}{p_{\g}}}\right)\,. 
\end{multline} \end{prop}

\begin{proof}
Let us fix $0 \leq T_{1} < T_{2} \leq T$. Integrating inequality \eqref{eq:13Ric} over $(t_{1},t_{2})$, we obtain that
\begin{multline*}
\frac{1}{2}\|f_{\l}^{+}(t_{2})\|_{L^{2}}^{2} + c_{0}\int_{t_{1}}^{t_{2}} \big\| \nabla \big(\langle \cdot \rangle^{\frac{\g}{2}}f_{\l}^{+}(\tau)\big) \big\|^{2}_{L^{2}} \d\tau \leq \frac{1}{2}\|f_{\l}^{+}(t_{1})\|_{L^{2}}^{2} \\
+\kappa_{0} \int_{t_{1}}^{t_{2}}\|\langle \cdot \rangle^{\frac{\g}{2}}f_{\l}^{+}(\tau)\|_{L^{2}}^{2} \d \tau - \l \int_{t_{1}}^{t_{2}}\d\tau\int_{\R^{d}}\bm{c}_{\g}[f](\tau,v)\,f_{\l}^{+}(\tau,v)\d v.
\end{multline*}
Thus, if $T_{1}\leq t_{1} \leq T_{2} \leq t_{2}\leq T$, one first notices that the above inequality implies that
\begin{multline*}
\frac{1}{2}\|f_{\l}^{+}(t_{2})\|_{L^{2}}^{2} + c_{0}\int_{T_{2}}^{t_{2}} \big\| \nabla \big(\langle \cdot \rangle^{\frac{\g}{2}}f_{\l}^{+}(\tau)\big) \big\|_{L^{2}}^{2} \d\tau \leq \frac{1}{2}\|f_{\l}^{+}(t_{1})\|_{L^{2}}^{2} \\
+\kappa_{0} \int_{T_{1}}^{t_{2}}\|\langle \cdot \rangle^{\frac{\g}{2}}f_{\l}^{+}(\tau)\|_{L^{2}}^{2}  \d \tau - \l \int_{T_{1}}^{t_{2}}\d\tau\int_{\R^{d}}\bm{c}_{\g}[f](\tau,v)\,f_{\l}^{+}(\tau,v)\d v,
\end{multline*}
and, taking the supremum over $t_{2} \in [T_{2},T]$, one gets
\begin{multline*}
\mathscr{E}_{\l}(T_{2},T) \leq \frac{1}{2}\|f_{\l}^{+}(t_{1})\|_{L^{2}}^{2}+\kappa_{0} \int_{T_{1}}^{T}\|\langle \cdot \rangle^{\frac{\g}{2}}f_{\l}^{+}(\tau)\|_{L^{2}}^{2} \d \tau\\
 - \l \int_{T_{1}}^{T}\d\tau\int_{\R^{d}}\bm{c}_{\g}[f](\tau,v)\,f_{\l}^{+}(\tau,v)\d v, \qquad \forall t_{1} \in [T_{1},T_{2}].\end{multline*}
Integrating now this inequality with respect to $t_{1} \in [T_{1},T_{2}]$, one obtains
\begin{multline*}
\mathscr{E}_{\l}(T_{2},T) \leq \frac{1}{2(T_{2}-T_{1})}\int_{T_{1}}^{T_{2}}\|f_{\l}^{+}(t_{1})\|_{L^{2}}^{2}\d t_{1}+ \kappa_{0}\int_{T_{1}}^{T}\|\langle \cdot \rangle^{\frac{\g}{2}}f_{\l}^{+}(\tau)\|_{L^{2}}^{2}  \d \tau\\
 - \l \int_{T_{1}}^{T}\d\tau\int_{\R^{d}}\bm{c}_{\g}[f](\tau,v)\,f_{\l}^{+}(\tau,v)\d v.\end{multline*}
Therefore, applying \cite[Proposition 2.4]{ABDL1} with $\lambda=\g <0$, $g=f$ and $\varphi = f^{+}_{\l}$, we see that
\begin{multline}\label{ef-f1}
\mathscr{E}_{\l}(T_{2},T) \leq \frac{1}{2(T_{2}-T_{1})}\int_{T_{1}}^{T}\|f_{\l}^{+}(\tau)\|_{L^{2}}^{2}\d \tau+\kappa_{0} \int_{T_{1}}^{T}\|\langle \cdot \rangle^{\frac{\g}{2}}f_{\l}^{+}(\tau)\|_{L^{2}}^{2} \d \tau\\
+ \ell \, C_{\g,p_{\g}}  {\sup_{\tau \in [T_{1},T]}\lm_{|\g|}(f(\tau))} \left(
\,\int_{T_{1}}^{T} \|\langle \cdot \rangle^{\g}f_{\l}^{+}(\tau)\|_{L^{1}} \d\tau + \,\int_{T_{1}}^{T} \|\langle \cdot \rangle^{\g}f_{\l}^{+}(\tau)\|_{L^{p_{\g}}} \d\tau \right),
\end{multline}
for $p_{\g} >1$ such that $-\g\,q_{\g} <  {d}$, where $\frac{1}{p_{\g}}+\frac{1}{q_{\g}}=1.$  We resort now to Lemma \ref{lem:changelevel} to estimate the last three terms on the right-hand side of \eqref{ef-f1}. Applying \eqref{eq:flL2}, one first has
\begin{equation*}
\begin{split}
\int_{T_{1}}^{T}\|\langle \cdot \rangle^{\frac{\g}{2}}f_{\l}^{+}(\tau)\|_{L^{2}}^{2}\ \d \tau &\leq C\,(\l-k)^{-\frac{4}{d}}\int_{T_{1}}^{T}\left\|\nabla \left(\langle \cdot \rangle^{\frac{\g}{2}}f_{k}^{+}(\tau)\right)\right\|_{L^{2}}^{2}\,\left\|f^{+}_{k}(\tau)\right\|_{L^{2}}^{ {\frac{4}{d}}} d\tau \\ 
&\leq \frac{C}{(\ell - k)^{{\frac{4}{d}}}}\sup_{t\in[T_{1},T]}\| f^{+}_{k}(t) \|^{ {\frac{4}{d}}}_{L^{2}}\int_{T_{1}}^{T}\left\|\nabla\left(\langle \cdot \rangle^{\frac{\g}{2}}f^{+}_{k}(\tau)\right)\right\|^{2}_{L^{2}} \d \tau .\\
\end{split}
\end{equation*}
Since 
$$\sup_{t\in [T_{1},T]}\|f_{k}^{+}(t)\|_{L^{2}}^{{\frac{4}{d}}} \leq \left(2\mathscr{E}_{k}(T_{1},T)\right)^{ {\frac{2}{d}}} \quad \text{ and } \quad \int_{T_{1}}^{T}\left\|\nabla\left(\langle \cdot \rangle^{\frac{\g}{2}}f^{+}_{k}(\tau)\right)\right\|^{2}_{L^{2}}\d \tau \leq c_{0}^{-1}\mathscr{E}_{k}(T_{1},T),$$ by definition of the energy functional, we get
\begin{equation}\label{eq:intl2}
\int_{T_{1}}^{T}\|\langle \cdot \rangle^{\frac{\g}{2}}f_{\l}^{+}(\tau)\|_{L^{2}}^{2} \d \tau \leq \bar{C}_{0}\,(\l-k)^{- {\frac{4}{d}}}\mathscr{E}_{k}(T_{1},T)^{ {\frac{d+2}{d}}}, 
\end{equation}
for some positive constant $\bar{C}_{0}$ depending only on $\|f_{\mathrm{in}}\|_{L^{1}_{2}}$ and $H(f_{\rm in})$. 
Similarly, using \eqref{eq:flLp} first with $p=1$ and then with $p=p_{\g}>1$, one deduces that
\begin{equation}\label{eq:intl1}\begin{split}
C_{\g,p_{\g}}(f_{\mathrm{in}})\,\int_{T_{1}}^{T} \|\langle \cdot \rangle^{\g}f_{\l}^{+}(\tau)\|_{L^{1}} \d\tau  &\leq \bar{C}_{0}(\ell -k)^{- {\frac{d+4}{d}}}\mathscr{E}_{k}(T_{1},T)^{ {\frac{d+2}{d}}}\,,\\
C_{\g,p_{\g}}(f_{\mathrm{in}})\,\int_{T_{1}}^{T} \|\langle \cdot \rangle^{\g}f_{\l}^{+}(\tau)\|_{L^{p_{\g}}} \d\tau  &\leq \bar{C}_{0}(\ell -k)^{-(\frac{2}{p_{\g}}- {\frac{d-4}{d}})}\mathscr{E}_{k}(T_{1},T)^{ {\frac{1}{p_{\g}}+\frac{2}{d}}}\,.
\end{split}
\end{equation} 
Regarding the first term in the right-hand side of  \eqref{ef-f1}, one uses \eqref{eq:flLq} with $q=\frac{2d+4}{d}+\frac{\g}{s}$ observing that $q \in \left(\frac{2d+2}{d},\frac{2d+4}{d}\right)$,  to get
\begin{multline*}
\int_{T_{1}}^{T}\|f^{+}_{\l}(\tau)\|_{L^{2}}^{2}\d \tau \leq \frac{c_{q}}{(\l-k)^{q-2}}\sup_{\tau\in [T_{1},T]}\|\langle \cdot \rangle^{s}\,f_{k}^{+}(\tau)\|_{L^{1}}^{ {\frac{2d+4}{d}}-q}\,\times\\
\times \int_{T_{1}}^{T}\|f_{k}^{+}(\tau)\|_{L^{2}}^{2(q- {\frac{2d+2}{d}})}\left\|\nabla \left(\langle \cdot \rangle^{\frac{\g}{2}}\,f_{k}^{+}(\tau)\right)\right\|_{L^{2}}^{2}\d \tau\\
\leq \frac{\bm{c}_{q}}{(\l-k)^{q-2}}\sup_{\tau\in [T_{1},T]}\|\langle \cdot \rangle^{s}\, f_{k}^{+}(\tau)\|_{L^{1}}^{ {\frac{2d+4}{d}}-q}\mathscr{E}_{k}(T_{1},T)^{q- {\frac{d+2}{d}}},
\end{multline*}
for some positive constant $\bm{c}_{q} >0.$ Thus
\begin{equation}\label{eq:f+lL2}
\int_{T_{1}}^{T}\|f^{+}_{\l}(\tau)\|_{L^{2}}^{2}\d \tau \leq \frac{\bm{c}_{q}}{(\l-k)^{q-2}}\left(\sup_{\tau \in [T_{1},T]}\lm_{s}(f(\tau))\right)^{ {\frac{2d+4}{d}}-q}\,\mathscr{E}_{k}(T_{1},T)^{q- {\frac{d+2}{d}}}.
\end{equation}
Gathering \eqref{ef-f1}--\eqref{eq:intl1}--\eqref{eq:intl2}--\eqref{eq:f+lL2} gives the result {recalling that $q= {\frac{2d+4}{d}}+\frac{\g}{s}$.}
\end{proof}
One deduces from this the following theorem.

\begin{theo} \label{Linfinito*} Under the assumptions of Proposition \ref{main-energy-functional}, let  $f(t,\cdot)$ be a weak-solution to \eqref{LFD}. Let us assume that $f_{\rm in}\in L^1_{s}(\R^d)$ for some $s >\frac{d}{2}|\g|$. Then, for any $T > t_{*} >0$,
\begin{equation}\label{eq:Linf}
\sup_{t \in [t_{*},T)}\left\|f(t)\right\|_{L^{\infty}} \leq C \,\mathscr{E}_{0}\left(\frac{t_{*}}{2},T\right)^{\alpha_1}\,  \left(\sup_{t \in [0,T]}\lm_{s}(f(t))\right)^{\alpha_2}\, \max\left(1,t_{*}^{-\frac{ds}{4s+d\g}}\right),
\end{equation}
for some explicit $C>0$, $\alpha_1$, $\alpha_2$.
\end{theo} 
\begin{proof} Let us first prove the appearance of the $L^{\infty}$-norm. In all the proof, for notations simplicity, we will denote by $C$ any positive constant depending on $f_{\rm in}$ which will change from line to line. Let us fix $T>t_{*}>0$ and let $K>0$ (to be chosen sufficiently large).  We consider the sequence of levels and times 
$$\l_{n}=K\,\left(1-\frac{1}{2^{n}}\right), \qquad t_{n}:=t_{*}\left(1-\frac{1}{2^{n+1}}\right), \qquad n \in \N.$$
We apply Proposition \ref{main-energy-functional} with the choices
$$k= \ell_{n} < \ell_{n+1}=\ell\,,\quad \quad T_{1}=t_{n}<t_{n+1} =T_{2}\,,\qquad {E}_{n}:=\mathscr{E}_{\l_{n}}(t_{n},T),$$
so that $\ell-k=K2^{-n-1},$ $T_{2}-T_{1}=t_{*}2^{-n-2}$ and we conclude that
\begin{multline*}
E_{n+1} \leq C t_{*}^{-1} 2^{n+2} K^{-\frac{4s+d\g}{ds}}2^{(n+1)\frac{4s+d\g}{ds}}\bm{y}_{s}^{\frac{|\g|}{s}}E _{n}^{1+\frac{2s+d\g}{ds}}\\
+C K^{1-\left(\frac{2}{p_{\g}}-\frac{d-4}{d}\right)}\bm{y}_{s}^{\frac{|\g|}{s}}2^{(n+1)\left(\frac{2}{p_{\g}}-\frac{d-4}{d}\right)}E_{n}^{\frac{1}{p_{\g}}+\frac{2}{d}}
+CK^{-\frac{4}{d}}\bm{y}_{s}^{\frac{|\g|}{s}}2^{(n+1)\left(1+\frac{4}{d}\right)} E_{n}^{1+\frac{2}{d}}\end{multline*}
where we set 
$\bm{y}_{s}=\sup_{t \in [0,T]}\lm_{s}(f(t))$
and used that $\l_{n} \leq K$ for any $n$ and
$\lm_{|\g|}(t) \le \bm{y}_{s}^{\frac{|\g|}{s}},$
by H\"older's inequality since $s>\frac{d}{2}|\g|\ge |\g|$ and $\|f(t)\|_{L^1}=1$ for any $t\ge0$. Let us note that the assumption $f_{\mathrm{in}}\in L^1_{s}(\R^d)$ together with Theorem \ref{theo:prop_mom}  ensure that $\bm{y}_{s}<\infty$. Then, rearranging the terms, we deduce that
\begin{multline}\label{DeG-ineq}
  {E}_{n+1}\leq C \bm{y}_{s}^{\frac{|\g|}{s}} t_{*}^{-1} K^{-\frac{4s+d\g}{ds}} 2^{n\left(1+\frac{4s+d\g}{ds}\right)} E_{n}^{1+\frac{2s+d\g}{ds}}  \\
  + C \bm{y}_{s}^{\frac{|\g|}{s}} K^{1-\frac{2}{p_{\g}}+\frac{d-4}{d}} 2^{n \left(\frac{2}{p_{\g}}-\frac{d-4}{d}\right) } E_{n}^{\frac{1}{p_{\g}}+\frac{2}{d}}
  + C\bm{y}_{s}^{\frac{|\g|}{s}} K^{-\frac{4}{d}} 2^{n\left(1+\frac{4}{d}\right)}E_{n}^{1+\frac{2}{d}}.
\end{multline}
Notice that
\begin{equation*} 
{E}_{0}=\mathscr{E}_{0}\left(\frac{t_{*}}{2},T\right)\leq \frac{1}{2}\sup_{t \in [\frac{t_{*}}{2},T)}\|f(t)\|_{L^{2}}^{2}+c_{0}\int_{\frac{t_{*}}{2}}^{T}\left\|\nabla \left(\langle \cdot\rangle^{\frac{\g}{2}}f(\tau)\right)\right\|_{L^{2}}^{2}\d\tau. \end{equation*}
We may apply Theorem \ref{theo:boundedL2} with $s=0$ and $p=2$ since, in this case, $\mu(0,2)= \frac{|\g|d}{4}<\frac{|\g|d}{2}$. Thereby, we deduce that
$$\sup_{t \in [\frac{t_{*}}{2},T)}\|f(t)\|_{L^{2}}^{2}<\infty.$$
Moreover, integrating \eqref{final-L2} over $\left(\frac{t_{*}}{2},T\right)$, still for $s=0$ and $p=2$, we obtain that
$$\int_{\frac{t_{*}}{2}}^{T}\left\|\nabla \left(\langle \cdot\rangle^{\frac{\g}{2}}f(\tau)\right)\right\|_{L^{2}}^{2}\d\tau <\infty.$$
Hence, $E_{0}$ is finite. We now look for a choice of the parameters  {$K$ and $Q >0$} ensuring that  the sequence $(E_{n}^{\star})_{n}$ defined by
$${E}^{\star}_{n}:={E}_0\,Q^{-n}, \qquad n \in \N\,,$$
satisfies \eqref{DeG-ineq} with the reversed inequality. Notice that 
\begin{multline}\label{deG-ineqstar}
{E}^{\star}_{n+1}\geq C \bm{y}_{s}^{\frac{|\g|}{s}} t_{*}^{-1} K^{-\frac{4s+d\g}{ds}} 2^{n\left(1+\frac{4s+d\g}{ds}\right)} \left(E_{n}^{\star}\right)^{1+\frac{2s+d\g}{ds}}  \\
  + C \bm{y}_{s}^{\frac{|\g|}{s}} K^{1-\frac{2}{p_{\g}}+\frac{d-4}{d}} 2^{n \left(\frac{2}{p_{\g}}-\frac{d-4}{d}\right) } \left(E_{n}^{\star}\right)^{\frac{1}{p_{\g}}+\frac{2}{d}}
  + C\bm{y}_{s}^{\frac{|\g|}{s}} K^{-\frac{4}{d}} 2^{n\left(1+\frac{4}{d}\right)}\left(E_{n}^{\star}\right)^{1+\frac{2}{d}},
\end{multline}
is equivalent to
\begin{multline*}
1\geq C \bm{y}_{s}^{\frac{|\g|}{s}} t_{*}^{-1} K^{-\frac{4s+d\g}{ds}} 2^{n\left(1+\frac{4s+d\g}{ds}\right)} E_{0}^{\frac{2s+d\g}{ds}} Q^{1-n\frac{2s+d\g}{ds}} \\
  + C \bm{y}_{s}^{\frac{|\g|}{s}} K^{1-\frac{2}{p_{\g}}+\frac{d-4}{d}} 2^{n \left(\frac{2}{p_{\g}}-\frac{d-4}{d}\right) } E_{0}^{\frac{1}{p_{\g}}+\frac{2}{d}-1}Q^{1-n\left(\frac{1}{p_{\g}}+\frac{2}{d}-1\right)}\\
  + C\bm{y}_{s}^{\frac{|\g|}{s}} K^{-\frac{4}{d}} 2^{n\left(1+\frac{4}{d}\right)}E_{0}^{\frac{2}{d}} Q^{1-n\frac{2}{d}},
\end{multline*}
which we can rewrite as 
\begin{multline*}
1\geq C \bm{y}_{s}^{\frac{|\g|}{s}} t_{*}^{-1} K^{-\frac{4s+d\g}{ds}} E_{0}^{\frac{2s+d\g}{ds}} \left[2^{1+\frac{4s+d\g}{ds}} Q^{-\frac{2s+d\g}{ds}} \right]^n\\
  + C \bm{y}_{s}^{\frac{|\g|}{s}} K^{2-\frac{2}{p_{\g}}-\frac{4}{d}} E_{0}^{\frac{1}{p_{\g}}+\frac{2}{d}-1} \left[ 2^{\frac{2}{p_{\g}}-\frac{d-4}{d}} Q^{-\left(\frac{1}{p_{\g}}+\frac{2}{d}-1\right)}\right]^n
  + C\bm{y}_{s}^{\frac{|\g|}{s}} K^{-\frac{4}{d}}E_{0}^{\frac{2}{d}} \left[ 2^{1+\frac{4}{d}}Q^{-\frac{2}{d}}\right]^n.
\end{multline*}
We first choose $Q$ in such a way that all the terms $\left[\cdots\right]^{n}$ are smaller than one, i.e.
$$Q=\max\left(2^{\frac{d+4}{2}},2^{\frac{4s+d(\g+s)}{2s+d\g}},2^{\frac{2d-(d-4)p_{\g}}{d-(d-2)p_{\g}}}\right), \qquad s >\frac{d}{2}|\g|\,.$$ With such a choice, \eqref{deG-ineqstar} would hold as soon as 
\begin{equation}\label{deG-ine1}
1\geq C \bm{y}_{s}^{\frac{|\g|}{s}} t_{*}^{-1} K^{-\frac{4s+d\g}{ds}} E_{0}^{\frac{2s+d\g}{ds}} \\
  + C \bm{y}_{s}^{\frac{|\g|}{s}} K^{2-\frac{2}{p_{\g}}-\frac{4}{d}} E_{0}^{\frac{1}{p_{\g}}+\frac{2}{d}-1} 
  + C\bm{y}_{s}^{\frac{|\g|}{s}} K^{-\frac{4}{d}}E_{0}^{\frac{2}{d}}.
  \end{equation}
This would hold for instance if each term of the sum is smaller than $\frac{1}{3}$, and  a direct computation shows that this amounts to choose
 {$$K \geq K(t_{*},T)=\max\left\{K_{1}(t_{*},T),K_{2}(t_{*},T),K_{3}(t_{*},T) \right\},$$
with
\begin{equation}\label{eq:Kt_{*}}
\begin{cases}K_{1}(t_{*},T)= C\,E_{0}^{\frac{2s+d\g}{4s+d\g}}\bm{y}_{s}^{\frac{d|\g|}{4s+d\g}}t_{*}^{-\frac{ds}{4s+d\g}}\,, \\
K_{2}(t_{*},T)= C \, E_{0}^{\frac{1}{2}}\bm{y}_{s}^{\frac{|\g| p_{\g}d}{2s(d -p_{\g}(d-2))}},\qquad \qquad K_{3}(t_{*},T)=C \,\bm{y}_{s}^{\frac{d|\g|}{4s}}\,E_{0}^{\frac{1}{2}}.\end{cases}
\end{equation}
By a comparison principle  (because ${E}_{0} ={E}^{\star}_{0}$), one concludes that $E_{n} \leq E_{n}^{\star}$ for all $n \in \N$ 
and in particular, since $Q >1$,
$\lim_{n}E_{n}=0.$ Since $\lim_{n}t_{n}=t_{*}$ and $\lim_{n}\ell_{n}=K$,  this implies that 
\begin{equation*}
\sup_{t\in[t_*,T)}\|f^{+}_{K}(t)\|_{L^{2}}= 0\,,
\end{equation*}
for $K \geq K(t_{*},T)$ and, in particular, 
\begin{equation*}\label{Linfinito}
\| f(t)\|_{L^{\infty}} \leq K(t_*,T)\,,\qquad  0<t_{*}\leq t < T\,.
\end{equation*}
Recall that {$K(t_{*},T)=\max\{K_{i}(t_{*},T),i=1,2,3\}$} as defined in \eqref{eq:Kt_{*}}. We estimate it roughly by the sum of these   terms, i.e.
$K(t_{*},T) \leq \sum_{i=1}^{3}K_{i}(t_{*},T)\,,$
and notice that the dependence with respect to $T$ and $t_*$ is encapsulated in the terms $E_{0}$, $t_{*}^{-1}$ and $\bm{y}_{s}$.}
\end{proof}

\subsubsection{Case $\g =-2$}

In that case, the fundamental result for the implementation of the level set analysis reads.
%\textcolor{red}{Maybe we can simplify to the worst homogeneity below}
\begin{prop}\label{main-energy-functional-2} Let $\g=-2$ and $T>0$. Let a nonnegative initial datum $f_{\mathrm{in}}\in L^1_2(\R^d) \cap L \log L (\R^d)$ satisfying \eqref{eq:Mass} be given. Let  $f(t,\cdot)$ be a weak-solution to \eqref{LFD}. Then, for any  $s > d$ and any $\alpha\in\left(\frac12,1\right)$, there exist some positive constants $ C_{1},C_{2},C_{3}$ depending only on $s$, $T$ and $f_{\rm in}$ such that, for any times $0 \leq T_{1} < T_{2} \leq T$ and $0 \leq k < \ell$, 
 \begin{multline}\label{eq:ElT2T3-2}
\mathscr{E}_{\l}(T_{2},T) \leq \frac{{C}_{1}}{T_{2}-T_{1}}(\l-k)^{{-\frac{4s-2d}{ds}}}\left[\sup_{\tau \in [T_{1},T]}\lm_{s}(f(\tau))\right]^{{\frac{2}{s}}}\,\mathscr{E}_{k}(T_{1},T)^{{\frac{(d+2)s-2d}{ds}}}
\\
+{C}_{2}\left(\l-k\right)^{-\frac4{d}} \mathscr{E}_{k}(T_{1},T)^{\frac{d+2}{d}}\\
+ C_{3}\l\,(\l-k)^{-2\alpha} \left(\sup_{t\in[T_1,T]}\lm_{\frac1{1-\alpha}}(f(t))\right)^{2(1-\alpha)} \mathscr{E}_{k}(T_{1},T)^{ 2\alpha}\,. 
\end{multline}
\end{prop}
 \begin{proof} Proceeding as in the proof of Proposition \ref{main-energy-functional}, we  still have
\begin{multline*}
\mathscr{E}_{\l}(T_{2},T) \leq \frac{1}{2(T_{2}-T_{1})}\int_{T_{1}}^{T_{2}}\|f_{\l}^{+}(t_{1})\|_{L^{2}}^{2}\d t_{1}+ \kappa_{0}\int_{T_{1}}^{T}\|\langle \cdot \rangle^{-1}f_{\l}^{+}(\tau)\|_{L^{2}}^{2}  \d \tau\\
- \l \int_{T_{1}}^{T}\d\tau\int_{\R^{d}}\bm{c}_{-2}[f](\tau,v)\,f_{\l}^{+}(\tau,v)\d v,\end{multline*}
and it follows from \eqref{eq:intl2} and \eqref{eq:f+lL2} that 
\begin{multline*}
  \mathscr{E}_{\l}(T_{2},T) \leq \frac{{C}_{1}}{T_{2}-T_{1}}(\l-k)^{{-\frac{4s-2d}{ds}}}\left[\sup_{\tau \in [T_{1},T]}\lm_{s}(f(\tau))\right]^{{\frac{2}{s}}}\,\mathscr{E}_{k}(T_{1},T)^{{\frac{(d+2)s-2d}{ds}}}\\
  +{C}_{2}\left(\l-k\right)^{-\frac4{d}} \mathscr{E}_{k}(T_{1},T)^{\frac{d+2}{d}} 
- \l \int_{T_{1}}^{T}\d\tau\int_{\R^{d}}\bm{c}_{-2}[f](\tau,v)\,f_{\l}^{+}(\tau,v)\d v,\end{multline*}
Therefore, it only remains to estimate the last term. First, by \cite[Lemma 2.2]{ABDL2}, we have
$$-\bm{c}_{-2}[f] \le -\bm{c}_{-2}[f_{\l}^{+}]$$
Then, we deduce from the Hardy-Littlewood-Sobolev inequality that
$$-\l \int_{\R^{d}}\bm{c}_{-2}[f](\tau,v)\,f_{\l}^{+}(\tau,v)\d v \le C \l \|f_{\l}^{+}\|_{L^{q}}\|f_{\l}^{+}\|_{L^{r}},\qquad \frac1{q}+\frac1{r}=2-\frac2{d}.$$
With $q=r=\frac{d}{d-1}$ and \eqref{eq:alphaf}, it becomes, for any  $0\le k <l$ and $\alpha \ge 0$, 
$$-\l \int_{\R^{d}}\bm{c}_{-2}[f](\tau,v)\,f_{\l}^{+}(\tau,v)\d v \le C \l \,(\l- k)^{-2\alpha} \,\|f_{k}^{+}\|_{L^{\frac{d(1+\alpha)}{d-1}}}^{2(1+\alpha)}.$$
Applying now \eqref{interp} with
$$q=\frac{d(1+\alpha)}{d-1}, \qquad \theta_1=\frac{1-\alpha}{1+\alpha}, \qquad \theta_2 = \frac{2\alpha-1}{1+\alpha}, \qquad \theta_3=\frac1{1+\alpha}, \qquad s=\frac1{1-\alpha}$$
and the Sobolev inequality, we obtain for any $\alpha \in\left(\frac12,1\right)$,
\begin{equation*}\begin{split}
  -\l \int_{\R^{d}}\bm{c}_{-2}[f](\tau,v)\,f_{\l}^{+}(\tau,v)&\d v   
  \le  C \l (\l- k)^{-2\alpha} \|\langle\cdot\rangle^{\frac1{1-\alpha}} f_{k}^{+}\|_{L^{1}}^{2(1-\alpha)}\|f_{k}^{+}\|_{L^{2}}^{2(2\alpha-1)}\|\langle\cdot\rangle^{-1} f_{k}^{+}\|_{L^{\frac{2d}{d-2}}}^{2} \\
  & \le   C \l (\l- k)^{-2\alpha} \|\langle\cdot\rangle^{\frac1{1-\alpha}} f_{k}^{+}\|_{L^{1}}^{2(1-\alpha)}\|f_{k}^{+}\|_{L^{2}}^{2(2\alpha-1)}\|\nabla\left(\langle\cdot\rangle^{-1}f_{k}^{+}\right)\|_{L^{2}}^{2}.
  \end{split}\end{equation*}
As in the proof of Proposition \ref{main-energy-functional}, it follows from the definition of the energy functional that
\begin{multline*}
  -\l \int_{T_{1}}^{T} \int_{\R^{d}}\bm{c}_{-2}[f](\tau,v)\,f_{\l}^{+}(\tau,v)\,\d v\, \d \tau \\\leq C_{3}\l\,(\l-k)^{-2\alpha} \left(\sup_{t\in[T_1,T]}\lm_{\frac1{1-\alpha}}(f(t))\right)^{2(1-\alpha)} \mathscr{E}_{k}(T_{1},T)^{ 2\alpha}, 
\end{multline*}
for some positive constant $C_{3}$. This completes the proof of \eqref{eq:ElT2T3-2}.
\end{proof} 
\begin{theo} \label{Linfinito*-2}Under the assumptions of  Proposition \ref{main-energy-functional-2}, let  $f(t,\cdot)$ be a weak-solution to \eqref{LFD}. Let us assume that $f_{\rm in}\in L^1_{s}(\R^d)$ for some $s >d$. Then, for any $T > t_{*} >0$,
\begin{equation}\label{eq:Linf-2}
\sup_{t \in [t_{*},T)}\left\|f(t)\right\|_{L^{\infty}} \leq C \,\mathscr{E}_{0}\left(\frac{t_{*}}{2},T\right)^{\alpha_1}\,  \left(\sup_{t\in[T_1,T]}\lm_{s}(f(t))\right)^{\alpha_2}\, \max\left(1,t_{*}^{-\frac{ds}{4s-2d}}\right),
\end{equation}
for some explicit $C>0$, $\alpha_1$, $\alpha_2$.
\end{theo} 
\begin{proof} We proceed here as in the proof of Theorem \ref{Linfinito*}. Therefore, we only give the main steps. Let us fix $T>t_{*}>0$ and let $K>0$ (to be chosen sufficiently large).  We consider the sequence of levels and times as in Theorem \ref{Linfinito*} and get
%$$\l_{n}=K\,\left(1-\frac{1}{2^{n}}\right), \qquad t_{n}:=t_{*}\left(1-\frac{1}{2^{n+1}}\right), \qquad n \in \N.$$
%We apply Proposition \ref{main-energy-functional-2} with the choices
%$$k= \ell_{n} < \ell_{n+1}=\ell\,,\quad \quad T_{1}=t_{n}<t_{n+1} =T_{2}\,,\qquad {E}_{n}:=\mathscr{E}_{\l_{n}}(t_{n},T),$$
%so that $\ell-k=K2^{-n-1},$ $T_{2}-T_{1}=t_{*}2^{-n-2}$ and we deduce that
\begin{multline}\label{DeG-ineq-2}
  {E}_{n+1}\leq C \bm{y}_{s}^{\frac{2}{s}} t_{*}^{-1} K^{-\frac{4s-2d}{ds}} 2^{n\left(1+\frac{4s-2d}{ds}\right)} E_{n}^{1+\frac{2s-2d}{ds}}  \\
  + C K^{-\frac{4}{d}} 2^{\frac{4n}{d}}E_{n}^{1+\frac{2}{d}}
  + C\bm{y}_{\frac{1}{1-\alpha}}^{2(1-\alpha)} K^{1-2\alpha} \,2^{2\alpha n}\,E_{n}^{2\alpha},
\end{multline}
where notations are those of Theorem \ref{Linfinito*} and  $\alpha$ to be chosen later on. As before, ${E}_{0}=\mathscr{E}_{0}\left(\frac{t_{*}}{2},T\right)$ is finite.
We now look for a choice of the parameters  {$K$ and $Q >0$} ensuring that  the sequence $(E_{n}^{\star})_{n}$ defined by
$${E}^{\star}_{n}:={E}_0\,Q^{-n}, \qquad n \in \N\,,$$
satisfies \eqref{DeG-ineq-2} with the reversed inequality. This amounts to satisfying
\begin{multline}\label{deG-ineqstar-2}
1\geq C \bm{y}_{s}^{\frac{2}{s}} t_{*}^{-1} K^{-\frac{4s-2d}{ds}} E_{0}^{\frac{2s-2d}{ds}} \left[2^{1+\frac{4s-2d}{ds}} Q^{-\frac{2s-2d}{ds}} \right]^n\\
  + C K^{-\frac{4}{d}}E_{0}^{\frac{2}{d}} \left[ 2^{\frac{4}{d}}Q^{-\frac{2}{d}}\right]^n
  + C\bm{y}_{\frac1{1-\alpha}}^{2(1-\alpha)} K^{1-2\alpha} E_{0}^{2\alpha-1} \left[ 2^{2\alpha}Q^{-(2\alpha-1)}\right]^n.
\end{multline}
We first choose $Q$ in  such a way that all the terms $\left[\cdots\right]^{n}$ are smaller than one, i.e.
$$Q=\max\left(2^2,2^{\frac{4s+d(s-2)}{2s-2d}},2^{\frac{2\alpha}{2\alpha-1}}\right), \qquad s >d \,, \quad \alpha\in\left( \frac12,1\right)\, .$$ With such a choice, \eqref{deG-ineqstar-2} would hold as soon as 
\begin{equation*}%\label{deG-ine1-2}
1\geq C \bm{y}_{s}^{\frac{2}{s}} t_{*}^{-1} K^{-\frac{4s-2d}{ds}} E_{0}^{\frac{2s-2d}{ds}}
  + C K^{-\frac{4}{d}} E_{0}^{\frac{2}{d}} 
  + C\bm{y}_{\frac1{1-\alpha}}^{2(1-\alpha)} K^{1-2\alpha}E_{0}^{2\alpha-1}.
  \end{equation*}
This would hold for instance if each term of the sum is smaller than $\frac{1}{3}$. Therefore, we choose
 $$K \geq K(t_{*},T)=\max\left\{K_{1}(t_{*},T),K_{2}(t_{*},T),K_{3}(t_{*},T) \right\}$$
with 
$$K_{1}(t_{*},T)= C\,E_{0}^{\frac{s-d}{2s-d}}\bm{y}_{s}^{\frac{d}{2s-d}}t_{*}^{-\frac{ds}{4s-2d}}, \quad K_{2}(t_{*},T)=C\sqrt{E_{0}}, \quad K_{3}(t_{*},T)=C \,\bm{y}_{s}^{\frac{2(1-\alpha)}{2\alpha-1}}\,E_{0}.$$
%$$K_{2}(t_{*},T)= C \, E_{0}^{\frac{1}{2}},\qquad \qquad K_{3}(t_{*},T)=C \,\bm{y}_{\frac1{1-\alpha}}^{\frac{2(1-\alpha)}{2\alpha-1}}\,E_{0}.$$
%\begin{equation*}\label{eq:Kt_{*}-2}
%\begin{cases}K_{1}(t_{*},T)= C\,E_{0}^{\frac{s-d}{2s-d}}\bm{y}_{s}^{\frac{d}{2s-d}}t_{*}^{-\frac{ds}{4s-2d}}\,, \\
%K_{2}(t_{*},T)= C \, E_{0}^{\frac{1}{2}},\qquad \qquad K_{3}(t_{*},T)=C \,\bm{y}_{\frac1{1-\alpha}}^{\frac{2(1-\alpha)}{2\alpha-1}}\,E_{0}.\end{cases}
%\end{equation*} 
Here above, since $s>d>2$, one may choose $\alpha\in\left(\frac12,1\right)$ such that
$\frac1{1-\alpha}=s$. As in the proof of Theorem \ref{Linfinito*}, we then conclude that \eqref{eq:Linf-2} holds, where $K(t_*,T)$ can be roughly  estimated as  $K(t_{*},T) \leq \sum_{i=1}^{3}K_{i}(t_{*},T)\,.$
\end{proof}

\appendix

\section{Known results about solutions to the Landau equation}\label{sec:known}

We collect here several mathematical known results about the solutions to the Landau equation in the range of parameters we are dealing with here, i.e. 
$$- 2\leq \g <0.$$
The results in this Appendix are meant to serve as a mathematical tool-box for the core of the paper. One begins with the following  coercivity estimate for the matrix $\mathcal{A}[f]$ 
\begin{prop}\label{diffusion}
Let $0\leq f_{\mathrm{in}}\in L^{1}_{2}(\R^{d}) \cap L\log L(\R^{d})$ be fixed and satisfying  \eqref{eq:Mass}.  Then, there exists a  constant $K_{0} > 0,$ depending on {$H(f_{\mathrm{in}})$ and $\|f_{\rm in}\|_{L^{1}_{2}}$}   such that
$$
 \sum_{i,j} \, \mathcal{A}_{i,j}[f](v) \, \xi_i \, \xi_j 
\geq K_{0} \langle v \rangle^{\g} \, |\xi|^2 , \qquad \forall\, v,\, \xi \in \R^{d},
$$
holds for any  $f \in \mathcal{Y}_{0}(f_{\mathrm{in}})$.
\end{prop}

We next recall the main result from \cite{CDH} concerning the propagation of polynomial moments of solutions to \eqref{LFD}. 
\begin{theo}[\textit{\textbf{Lemma 7 of \cite{CDH}}}]\label{theo:prop_mom} Assume that $-2\le \gamma<0$. Let a nonnegative initial datum $f_{\rm in} \in L^{1}_{s}  {\cap L\log L}(\R^{d})$ $(s>2)$ be given satisfying \eqref{eq:Mass} and consider any global weak solution $f(t,v)$ to \eqref{LFD} with initial datum $f_{\rm in}.$ Then there exists $\bm{C}_{s,\g} >0$ depending on $s,\g$, and $\|f_{\rm in}\|_{L^{1}_{2}}$ such that
$$\lm_{s}(t) \leq \lm_{s}(0)+ \bm{C}_{s,\g} t, \qquad t \geq 0\,.$$
\end{theo}

\section{Elements of Lorentz spaces}\label{sec:Lorentz}

We collect here elementary properties of Lorentz spaces that are used in the core of the paper.  A main  reference for the results is \cite{grafakos}. Let $(X,\mathcal{F},\mu)$ be a given measure space. In practice, we consider the case $X=\R^{d}$ endowed with the Borel $\sigma$-algebra and the Lebesgue measure. For a measurable mapping
$$f\::\:X \to \R$$
define the distribution
$${d}_{f}(s):=\mu\left(\left\{x \in \R^{d}\;;\;|f(x)| > s\right\}\right), \qquad s \in \R^{+}.$$
%\begin{exa}\label{exa:1} If $f$ is a nonnegative simple function
%$$f=\sum_{i=1}^{N}\alpha_{i}\ind_{E_{i}}, \qquad E_{i} \in \mathcal{F}$$
%with $\alpha_{1} > \alpha_{2} > \ldots,\alpha_{N} >0$. Then, setting
%$B_{k}=\sum_{j=1}^{k}\mu(E_{j})$,
%one has
%$$d_{f}(s)=\sum_{k=0}^{N}B_{k}\ind_{[a_{k+1},a_{k})}(s), \qquad s \in \R^{+}$$
%with the convention $a_{N+1}=0,$ $B_{0}=0$.
%\end{exa}
%Recall that, for any $\varphi\::\:\R^{+}\to\R^{+}$ differentiable, $\varphi(0)=0$, one has
%$$\int_{0}^{\infty}\varphi'(s)\,d_{f}(s)\,\d s=\int_{X}\varphi\left(|f|\right)\d\mu$$
%thanks to a simple use of Fubini's theorem. In particular, given $p \geq 1$, 
%$$\|f\|_{L^{p}}^{p}=p\int_{0}^{\infty}s^{p-1}d_{f}(s)\,\d s\,.$$
%thanks to a simple use of Fubini's theorem. More generally, for any $\varphi\::\:\R^{+}\to\R^{+}$ differentiable, $\varphi(0)=0$, 
%$$\int_{0}^{\infty}\varphi'(s)\,d_{f}(s)\,\d s=\int_{X}\varphi\left(|f|\right)\d\mu.$$
We introduce then the \emph{decreasing rearrangement of} $f$ as the function 
$$f^{\ast}\::\:\R^{+}\to\R^{+}$$
defined as
$$f^{\ast}(t):=\inf\left\{s >0\;;\;d_{f}(s) \leq t\right\}, \qquad t \geq 0.$$
One has $f^{\ast}$ is nonincreasing and supported in $[0,\mu(X))$. Moreover, 
$$\left(|f|^{p}\right)^{\ast}=\left(f^{\ast}\right)^{p}, \qquad 0 < p < \infty$$
and
$$\|f\|_{L^{p}}^{p}=\int_{0}^{\infty}\left(f^{\ast}\right)^{p}(t)\d t, \qquad \|f\|_{L^{\infty}}=f^{\ast}(0).$$

\begin{defi} Let $f\::\:X \to \R$ be measurable and $1\leq p,q \leq \infty$. One defines
\begin{equation}\label{lorentz}
\|f\|_{p,q}:=\begin{cases} \left(\displaystyle\int_{0}^{\infty}\left(t^{\frac{1}{p}}f^{\ast}(t)\right)^{q}\dfrac{\d t}{t}\right)^{\frac{1}{q}}, \quad \qquad &q < \infty\,,\\
\\
\sup_{t >0}t^{\frac{1}{p}}f^{\ast}(t), \quad \qquad &q=\infty.\end{cases}
\end{equation}
The Lorentz space $L^{p,q}(X,\mu)$ is the space of all measurable $f\::\:X \to \R$ for which $\|f\|_{p,q} <\infty.$
\end{defi}
\begin{rmq} One has
$$\left\|\,|f|^{r}\right\|_{p,q}=\left\|f\right\|_{pr,qr}^{r}, \qquad r >0.$$
%while, if $X=\R^{d}$, given $\delta >0$, one can define the dilation 
%$$f_{\delta}(x)=f(\delta\,x)$$
%and sees that $d_{f_{\delta}}(s)=\delta^{-d}d_{f}(s),$ $s >0$ and $f_{\delta}^{\ast}(t)=f^{\ast}(\delta^{d}t)$ so that
%$$\|f_{\delta}\|_{p,q}=\delta^{-\frac{d}{p}}\,\|f\|_{p,q}.$$
Moreover,  $\|f\|_{p,q}=0$ if and only if $f=0$ $\mu$-a.e. on $X$ and it can be shown that
\begin{equation}\label{eq:ppqq}
\|f\|_{p,q}^{q}=p\int_{0}^{\infty}\left[s\,d_{f}(s)^{\frac{1}{p}}\right]^{q}\dfrac{\d s}{s}.\end{equation}
If $X=\R^{d}$ is endowed with the Borel $\sigma$-algebra and the Lebesgue measure, we simply denote the corresponding $L^{p,q}(X,\mu)$ space as  $L^{p,q}(\R^{d})$.
\end{rmq} 
\begin{rmq} Notice that, for $1 \leq p \leq \infty$ and $1 \leq q < \infty$, the space $\left(L^{p,q}(X,\mu),\|\cdot\|_{p,q}\right)$ is then a quasi-Banach space, i.e. it is complete for the \emph{quasi-norm} $\|\cdot\|_{p,q}$.  Moreover, 
$$L^{p,p}(X,\mu)=L^{p}(X,\mu),\qquad \|f\|_{L^{p}}=\|f\|_{p,p}, \qquad \forall f \in L^{p}(X,\mu).$$
\end{rmq}

A version of H\"older inequality is known to hold in Lorentz spaces.
\begin{prop} Let $1 < p < \infty$ and $1 \leq q \leq \infty$. We define $p',q'$ by $\frac{1}{q}+\frac{1}{q'}=1,$ $\frac{1}{p}+\frac{1}{p'}=1$. If $f \in L^{p,q}(X,\mu)$ and $g \in L^{p',q'}(X,\mu)$ then $fg \in L^{1}(X,\mu)$ and
$$\left|\int_{X}fg\,\d\mu\right| \leq \|f\|_{p,q}\|g\|_{p',q'}\,.$$
Moreover, if $1 \leq p_{1} < p < p_{2} < \infty$,
\begin{equation}\label{eq:interpopq}\|f\|_{p,q} \leq \frac{p^{\frac{1}{q}}}{p_{1}^{\theta}p_{2}^{1-\theta}}\|f\|_{p_{1},q}^{\theta}\,\|f\|_{p_{2},q}^{1-\theta}, \qquad \tfrac{1}{p}=\tfrac{\theta}{p_{1}}+\tfrac{1-\theta}{p_{2}}.\end{equation}
\end{prop}
We also recall the following refined version of Sobolev inequality. 
\begin{theo}\label{theo:Sob} For $d \geq 3$ and $1 \leq q < d$, there exists $C_{d,q} >0$ such that
$$\|f\|_{q^{*},q} \leq C_{d,q}\,\|\nabla f\|_{L^{q}}, \qquad q^{*}=\frac{qd}{d-q}\,,$$
for any compactly supported function $f\::\R^{d}\to\R$ and where $\|\cdot\|_{q^{*},q}$ denotes the (quasi)-norm on $L^{q^{*},q}(\R^{d})$.\end{theo}

\end{document}